\documentclass{amsart}
\usepackage{amsmath, amsthm, amsfonts, amssymb}
\usepackage{mathrsfs}
\usepackage{microtype}
\usepackage[utf8]{inputenc}
\providecommand{\noopsort[1]{}}
\usepackage{bbm}
\usepackage{dsfont}
\usepackage{ifthen}
\usepackage{nicefrac}
\numberwithin{equation}{section}
\usepackage{a4}
\usepackage[active]{srcltx}
\usepackage{verbatim}
\usepackage{hyperref,color,tikz, xcolor}

\usepackage[shortlabels]{enumitem}
\setlist{leftmargin=*}
\setlist[1]{labelindent=1.2\parindent}

\newcommand\pgfmathsinandcos[3]{%
  \pgfmathsetmacro#1{sin(#3)}%
  \pgfmathsetmacro#2{cos(#3)}%
}
\newcommand\LongitudePlane[3][current plane]{%
  \pgfmathsinandcos\sinEl\cosEl{#2} 
  \pgfmathsinandcos\sint\cost{#3} 
  \tikzset{#1/.estyle={cm={\cost,\sint*\sinEl,0,\cosEl,(0,0)}}}
}
\newcommand\LatitudePlane[3][current plane]{%
  \pgfmathsinandcos\sinEl\cosEl{#2} 
  \pgfmathsinandcos\sint\cost{#3} 
  \pgfmathsetmacro\yshift{\cosEl*\sint}
  \tikzset{#1/.estyle={cm={\cost,0,0,\cost*\sinEl,(0,\yshift)}}} %
}

\tikzset{%
  >=latex, 
  inner sep=0pt,%
  outer sep=2pt,%
  mark coordinate/.style={inner sep=0pt,outer sep=0pt,minimum size=3pt,
    fill=black,circle}%
}

\newtheorem{thm}{Theorem}[section]
\newtheorem{cor}[thm]{Corollary}
\newtheorem{prop}[thm]{Proposition}
\newtheorem{lem}[thm]{Lemma}

\theoremstyle{remark}
\newtheorem{rem}[thm]{Remark}

\theoremstyle{definition}
\newtheorem{defn}[thm]{Definition}
\newcommand{\coloneqq}{\mathrel{\mathop:}=}

\renewcommand{\Re}{{\rm Re}\,}

\renewcommand{\Im}{{\rm Im}\,}
\newcommand{\eps}{\varepsilon}

\newcommand{\one}{\mathds{1}}

\newcommand{\CR}{\mathds{R}}
\newcommand{\CC}{\mathds{C}}

\newcommand{\CN}{\mathds{N}}

\newcommand{\ha}{\mathfrak {H}}

\newcommand{\e}{\mathrm {e}}

\newcommand {\x}{\mathbb{X}}

\newcommand{\mcj}{J}

\newcommand{\cL}{\mathscr{L}}

\newcommand{\weak}{\rightharpoonup}
\renewcommand{\eps}{\varepsilon}
\newcommand{\ud}{\, \mathrm{d}}
\newcommand{\lam}{\lambda}
\newcommand{\dlam}{\ud \lam }
\newcommand{\la}{\langle}
\newcommand{\ra}{\rangle}
\newcommand{\mc}{\mathscr}

\newcommand{\grak}{\lim_{\kappa \to \infty}}

\DeclareMathOperator{\interior}{\mathrm{Int}}
\let\div\undefined
\DeclareMathOperator{\div}{\mathrm{div}}
\DeclareMathOperator{\fa}{\mathfrak{h}}      
\DeclareMathOperator{\fb}{\mathfrak{q}}    
\DeclareMathOperator{\fq}{\mathfrak{a}}    

\DeclareMathOperator{\sgn}{\mathrm{sgn}}
\DeclareMathOperator{\diag}{\mathrm{diag}}

\begin{document}
\title[An averaging principle for fast diffusions]{An averaging principle for fast diffusions in domains separated by semi-permeable membranes}
\author{Adam Bobrowski}
\email{a.bobrowski@pollub.pl}
\address{Lublin University of Technology, Nadbystrzycka 38A, 20-618 Lublin, Poland}
\author{Bogdan Kazmierczak} 
\email{bkazmier@ippt.pan.pl}
\address{Institute of Fundamental Technological Research, PAS, Pawinskiego 5B, 02-106 Warsaw, Poland}
\author{Markus Kunze}
\email{markus.kunze@uni-konstanz.de}
\address{Universit\"at Konstanz, Fachbereich Mathematik und Statistik, 78467 Konstanz, Germany}

\begin{abstract}
We prove an averaging principle which asserts convergence of diffusion processes on domains separated by semi-permeable membranes, when diffusion coefficients tend to infinity while the flux through the membranes remains constant. In the limit, points in each domain are lumped into a single state of a limit Markov chain. The limit chain's intensities are proportional to the membranes' permeability and inversely proportional  to the domains' sizes. 
Analytically, the limit is an example of a singular perturbation in which boundary and transmission conditions play a crucial role.  This averaging principle is strongly motivated by recent signaling pathways models of mathematical biology, which are discussed towards the end of the paper. 
\end{abstract}

\keywords{Convergence of sectorial forms and of semigroups of operators, diffusion processes, boundary and transmission conditions, Freidlin--Wentzell averaging principle, singular perturbations, signaling pathways, kinase activity, intracellular calcium dynamics, neurotransmitters}
\subjclass[2010]{47A07, 47D07, 60J70, 92C45}

\maketitle 

\section{Introduction}  
The main aim of this article is to establish an averaging principle saying that fast diffusion processes on domains separated by semi-permeable domains may be approximated by certain Markov chains. More specifically, if diffusion's speed in each domain increases while the flux through the boundaries remains constant, in the limit, all points in each domain are lumped together to form a single state, and the limit process is a Markov chain whose state-space is composed of these lumped states (Theorems \ref{t.l2conv} and \ref{t.lpconv}). The jump intensities in the chain are  in direct proportion to the total permeability of the membranes, and in inverse proportion to the sizes of the domains (see eq.\ \eqref{qkl}). We note that the principle just described is akin to the famous Freidlin--Wentzell averaging principle (\cite{fw,fwbook}, see also \cite{freidlin}), though it is motivated by biological rather than physical models. Moreover, in contrast to the Freidlin--Wentzell principle, in our case the crucial role is played by transmission conditions.  

Predecessors of our principle have been studied in \cite{bobmor} and \cite{nagrafach}, see also \cite{gregosiewicz}.  In \cite{bobmor}, in an attempt to reconcile two models of so-called neurotransmitters (a macroscopic one of Aristizabal and Glavinovi\v{c} \cite{ag}, and a microscopic one of Bielecki and Kalita \cite{bielecki}) it has been shown that fast diffusions in three domains, corresponding to the so-called large, small, and immediately available pools, may be approximated by a Markov chain with three states, see Figure \ref{figure2}. In fact, in \cite{bobmor} merely a one-dimensional variant of this limit theorem has been proved, in which the three 3-dimensional pools are replaced by three adjacent intervals. This result has later been generalized to the case of fast diffusions on arbitrary finite graphs in \cite{nagrafach}; in both cases the limit theorems are stated as convergence theorems for semigroups in Banach spaces of continuous functions. In \cite{gregosiewicz}, a related result has been proved in a space of integrable functions.
See also \cite{banfalnam} for a generalization. 

\begin{figure}
\begin{tikzpicture}[scale=0.8]
\def\R{2} 
\def\x{0.7*\R}
\def\y{0.5*\x}
\def\angEl{5} 
\def\angAz{-125} 
\def\angPhi{0} 
\def\angBeta{90} 
\def\du{pink!20}
\def\srramka{red}
\def\sr{\srramka!30}
\def\maramka{brown}
\def\ma{\maramka!70}
\pgfmathsetmacro\H{\R*cos(\angEl)} 
\tikzset{xyplane/.estyle={cm={cos(\angAz),sin(\angAz)*sin(\angEl),-sin(\angAz),
                              cos(\angAz)*sin(\angEl),(0,0)}}}
\LongitudePlane[xzplane]{\angEl}{\angAz}
\LongitudePlane[pzplane]{\angEl}{\angPhi}
\LatitudePlane[equator]{\angEl}{0}

\shade[left color=brown,shading=axis,shading angle=220,] (0,0) ellipse(\x in and \y in);

\coordinate (N) at (0,2.55*\y);
\coordinate (S) at (0,-2.55*\y);



\fill[fill=\du,scale=0.9*\R] (0,0)-- (N)  arc (90:-90: 2 and 1)--(0,0);
\filldraw[fill=\du,draw=black,very thin,scale=0.89*\R] (0,0)--(N)arc (90:270: 0.7 and 1)--(0,0)  ;
\filldraw[fill=\sr, draw=\srramka,scale=1.7*\y] (0,0)-- (0,1.4*\y)  arc (90:-90: 2 and 1)--(0,0);
\filldraw[fill=\sr,draw=\srramka,scale=1.7*\y] (0,0)--(0,1.4*\y)  arc (90:270: 0.7 and 1)--(0,0);
\fill[fill=\ma,draw=\maramka,scale=1.3*\y] (0,0)--  (0,1.15*\y)  arc (90:-90: 2 and 1)--(0,0);
\fill[fill=\ma,draw=\maramka,scale=1.3*\y] (0,0)-- (0,1.15*\y)  arc (90:270: 0.7 and 1)--(0,0);

\draw[-, very thin] (N) -- (S);
\draw[very thin] (0,0) ellipse (\x in and \y in);


\coordinate(si1) at (2.25*\x,0*\y); 
\coordinate(si2) at (1.4*\x,0*\y); 
\coordinate(si3) at (0.2*\x,0*\y); 

\coordinate(na1) at (2.25*\x,-3.5*\y); 
\coordinate(na2) at (1.4*\x,+3.5*\y); 
\coordinate(na3) at (0.2*\x,-3.5*\y);

\coordinate(na4) at (-1.6*\x,3.5*\y);
\coordinate(na5) at (2*\x,3.5*\y);
\coordinate(na6) at (-1.6*\x,4.2*\y);

\draw[->, very thin,dashed] (na1)node[right]{\large{Immediately available pool}} -- (si1); 
\draw[->, very thin, dashed] (na2)node[right]{\large{Small pool}} -- (si2); 
\draw[->, very thin, dashed] (na3)node[right]{\large{Large pool}} -- (si3);\draw[color=black,->] (5,0) -- (9,0) node [pos=0.5,above] {fast diffusion} node [pos=0.5,below] {constant flux};
\draw[fill=\ma] (11,-1) circle (0.15); \draw[fill=\sr] (11,0) circle (0.15); \draw[fill=\du] (11,1) circle (0.15); 

\end{tikzpicture} \caption{Averaging principle in two models of neurotransmitters} \label{figure2}\end{figure}
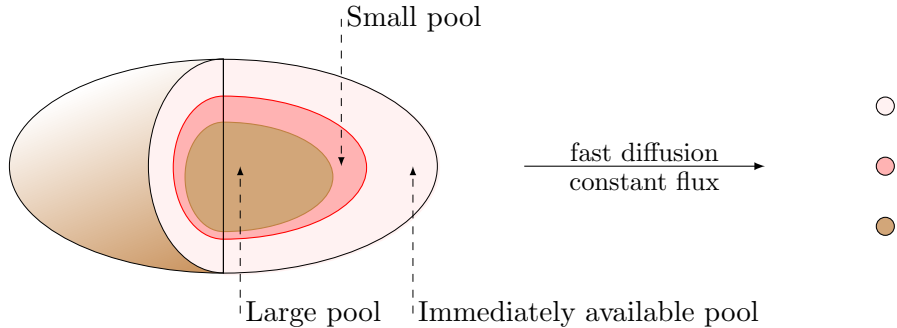

In this paper, we come back to the general, $d$-dimensional setting, and, in contrast to the previous papers, focus on the analysis in $L^p$ spaces ($p\ge 1$). {At first, we prove our main result in $L^2$  (see Section 
\ref{sect.conv} and
Theorem \ref{t.l2conv} in particular), using convergence theorems for quadratic forms. Later on we extend the analysis to other $L^p$ spaces} (Theorem \ref{t.lpconv}) by extrapolation and interpolation techniques -- see Section \ref{extension}. In Section \ref{sec:6}, three contemporary models of biology, including two recent signaling pathways models and the neurotransmitters model, are discussed as special cases of the principle so established.

As already mentioned, the key role in this analysis is played by transmission conditions (see \eqref{eq.trans}) describing the permeability of membranes.
In the context of heat flow,  these conditions may be plausibly interpreted: according to Newton's Law of Cooling, the temperature at the membrane changes at a rate proportional to the difference of temperatures on  
either sides of the membrane, see \cite[p.\ 9]{crank}. In this context, J.\ Crank uses the term \emph{radiation boundary condition}. (Although, strictly speaking, these are not \emph{boundary}, but \emph{transmission} conditions, see \cite{lions2,lions1,lions3}.) 

In the context of passing or diffusing through membranes, analogous transmission conditions were introduced by J.\ E.\ Tanner \cite[eq.\ (7)]{tanner}, who studied diffusion of particles through a sequence of permeable barriers (see also Powles et al. \cite[eq.\ (1.4)]{powles}, for a continuation of the subject).  In \cite{andrews} (see e.g.\ eq.\ (4) there) similar conditions are used in describing absorption and desorption phenomena. We refer also to \cite{fireman}, where a compartment model with permeable walls (representing e.g., cells, and axons in the white matter of the brain in particular) is analyzed, and to equation [42] there.  

In the context of neurotransmitters, conditions of type \eqref{eq.trans} were (re)-invented in  \cite{bobmor} and \cite{nagrafach}, interpreted in probabilistic terms, and linked with Feller--Wentzell's boundary conditions \cite{fellera1,fellera2,fellera4,fellera3,wentzell} (see 
\cite{lejay} for a more thorough stochastic analysis). 

A systematic study of semigroups and cosine families related to such transmission conditions has been commenced in \cite{tombatty}. We note also the recent paper \cite{bardos}, where a heat problem for such transmission conditions is studied for quite irregular boundaries, and the monograph \cite{agranovicz} in which related transmission conditions are analyzed.

\section{General idea and a word on mathematical tools} 

\subsection{General idea} \label{s.gi}
One of the fundamental properties of diffusion in a bounded domain is that it `averages' solutions of the heat equation over the domain (see e.g. \cite{smoller}). The effect is probably best known in the context of Neumann boundary conditions. To restrict ourselves to the simplest one-dimensional case, consider the heat equation in the interval $[a,b]$ ($a< b$)
\begin{equation} \frac {\partial u (t,x)}{\partial t} = \kappa \frac {\partial^2 u (t,x)}{\partial x^2}, \qquad x \in (a,b),  \label{simplemodel}\end{equation}
where $\kappa >0$ is a diffusion constant, with Neumann boundary conditions $\frac {\partial u (t,a)}{\partial x} = \frac {\partial u (t,b)}{\partial x} =0$. Then, regardless of the choice of the initial condition, say a continuous function $u_0$ on $[a,b]$,
\begin{equation}\lim_{t\to \infty}  u(t,x) = \frac 1{b-a} \int_a^b u_0(y)  \ud y , \label{intro1}\end{equation}
 uniformly with respect to $x\in [a,b]$. The same is true also when $u_0$ is a member of $L^p(a,b), p\ge 1$,  the space of  functions on $(a,b)$  that are
absolutely integrable with $p$-th power, and convergence is understood in the sense of the norm in this space. (In $L^2(a,b)$ this can be demonstrated by looking at the Fourier expansion of the solution.) A physical interpretation of \eqref{intro1} is that as time passes the temperature distribution in an isolated finite rod `averages out' and becomes constant throughout the rod. 

In modeling biological phenomena, one sometimes may assume that diffusion involved in the model is of several magnitudes faster than other processes. This leads to a study of the situation where diffusion coefficient(s) converge(s) to infinity. For example, 
in our simple model \eqref{simplemodel}, we  could be interested in letting $\kappa \to \infty $. A counterpart of \eqref{intro1} would than say that, if heat is propagated in the rod without hindrances, then `before other forces intervene' the temperature will become constant throughout the rod. So, if the rod is a part of a larger system, its state at time $t>0$ may in fact be described by a single number, and not by a comparatively more complex object, i.e., a temperature distribution function.

To look at a more interesting situation, consider $a<0<b$ and diffusion in two adjacent intervals, $(a,0)$ and $(0,b),$ separated by semi-permeable membrane at $x=0$. Assuming, for simplicity, that diffusion coefficients in both intervals are the same, we write  the diffusion equation 
\begin{equation} \frac {\partial u (t,x)}{\partial t} = \kappa \frac {\partial^2 u (t,x)}{\partial x^2}, \qquad x \in (a,0)\cup (0,b),  \label{simplenotmodel}\end{equation}
and impose Neumann boundary conditions \[\frac {\partial u (t,a)}{\partial x} = \frac {\partial u (t,b)}{\partial x} =0\] at the intervals' ends. The membrane at $x=0$ is characterised by the following pair of transmission conditions: 
\begin{equation} \label{transmit}
\frac {\partial u (t,0+)}{\partial x} = \frac {\partial u (t,0-)}{\partial x}, \quad \kappa \frac {\partial u (t,0+)}{\partial x}  = \beta u(t,0+) - \alpha u (t,0-),
 \end{equation}
where $\alpha $ and $\beta $ are positive parameters to be described below. To explain the meaning of these conditions, we interpret $u$ as a distribution of temperature throughout the rod, and 
introduce the total sum of temperatures at the right part of the rod:
$v_+(t) = \int_0^b u(t,x) \ud x $. Then 
\[ \frac {\ud v_+ (t)}{\ud t} =  \kappa \int_0^b \frac {\partial^2 u (t,x)}{\partial x^2}\ud x = - \kappa \frac {\partial u (t,0+)}{\partial x} . \]
Since a similar calculation shows that, for $v_-(t) = \int_{a}^0 u(t,x)\ud x$, we have
\[ \frac {\ud v_-(t)}{\ud t} = \kappa \frac {\partial u (t,0-)}{\partial x} , \]
the first equation in \eqref{transmit} transpires to be a balance condition: the amount of heat lost or gained by one part of the rod is the amount of heat gained or lost by the  other.

Furthermore, assume for a moment that $\alpha $ in \eqref{transmit} is zero. Then the second equation there becomes all-familiar Robin boundary condition for diffusion on $(0,b)$ with partial heat loss at $x=0$. Then, both conditions combined can be interpreted by saying that some particles diffusing in $(0,b)$ may permeate through the membrane and thus transfer heat from the right interval to the left. The larger is $\beta$, the larger is the heat loss at the membrane (as seen from the perspective of the right interval), and thus the larger is in fact the heat transfer from the right to the left. Hence, $\beta$ is a permeability coefficient for the membrane and describes the possibility for particles to filter through the membrane from the right interval to the left interval. An analogous statement is true about $\alpha$: it characterises permeability of the membrane when approached from the left.

The main point, though, is that as $\kappa \to  \infty$, temperatures at both parts of the rod will average out, so that in the limit, $u(t,0+) $ and $u(t,0-)$ may be replaced by $b^{-1} v_+(t)$ and $|a|^{-1} v_-(t)$, respectively, yielding
\begin{align} 
\frac {\ud v_- (t)}{\ud t} &= - \alpha |a|^{-1} v_-(t) + \beta b^{-1} v_+(t), 
\label{ukladzik} \\
\frac {\ud v_+ (t)}{\ud t} &= \phantom{-} \alpha |a|^{-1} v_-(t)-  \beta b^{-1} v_+(t). 
\nonumber
\end{align}
Remarkably, these equations describe transient probabilities  of a Markov chain with two states, say $-$ and $+$; this chain starting at state $+$ spends an exponential time there with parameter $\beta b^{-1}$, and then jumps to the state $-$. While at state $-$, it forgets its past and waits for independent exponential time with parameter $\alpha |a|^{-1}$ before jumping to $+,$ and so on. Quantity $v_+ (t)$ is then the probability that at time $t\ge0$ the chain is at state $+$, and $v_-(t)$ is the probability that it is at state $-$. 

Our analysis illustrates that speeding up diffusion in two intervals while decreasing permeability of the membrane in such a way that the flux through the membrane remains constant, leads to the limit in which heat conduction is  modeled by two state Markov chain: `heat' is gathered in two containers and its `particles' may jump `over the membrane' from one container to the other. These intuitions are supported by simulations (see Figure \ref{fig1}). 
\begin{figure}[ht]
\includegraphics[width=0.45 \linewidth]{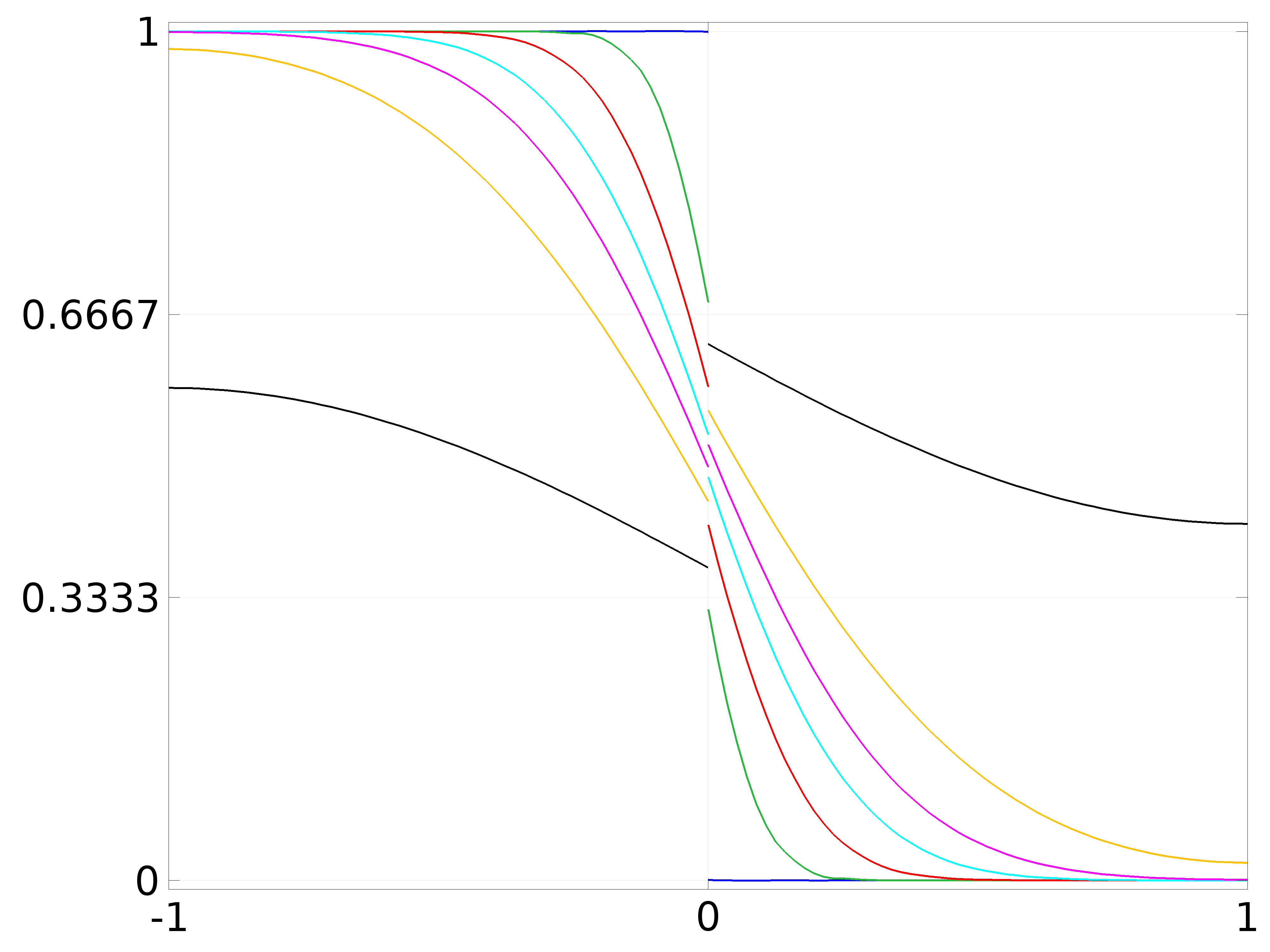}  
\includegraphics[width=0.45 \linewidth]{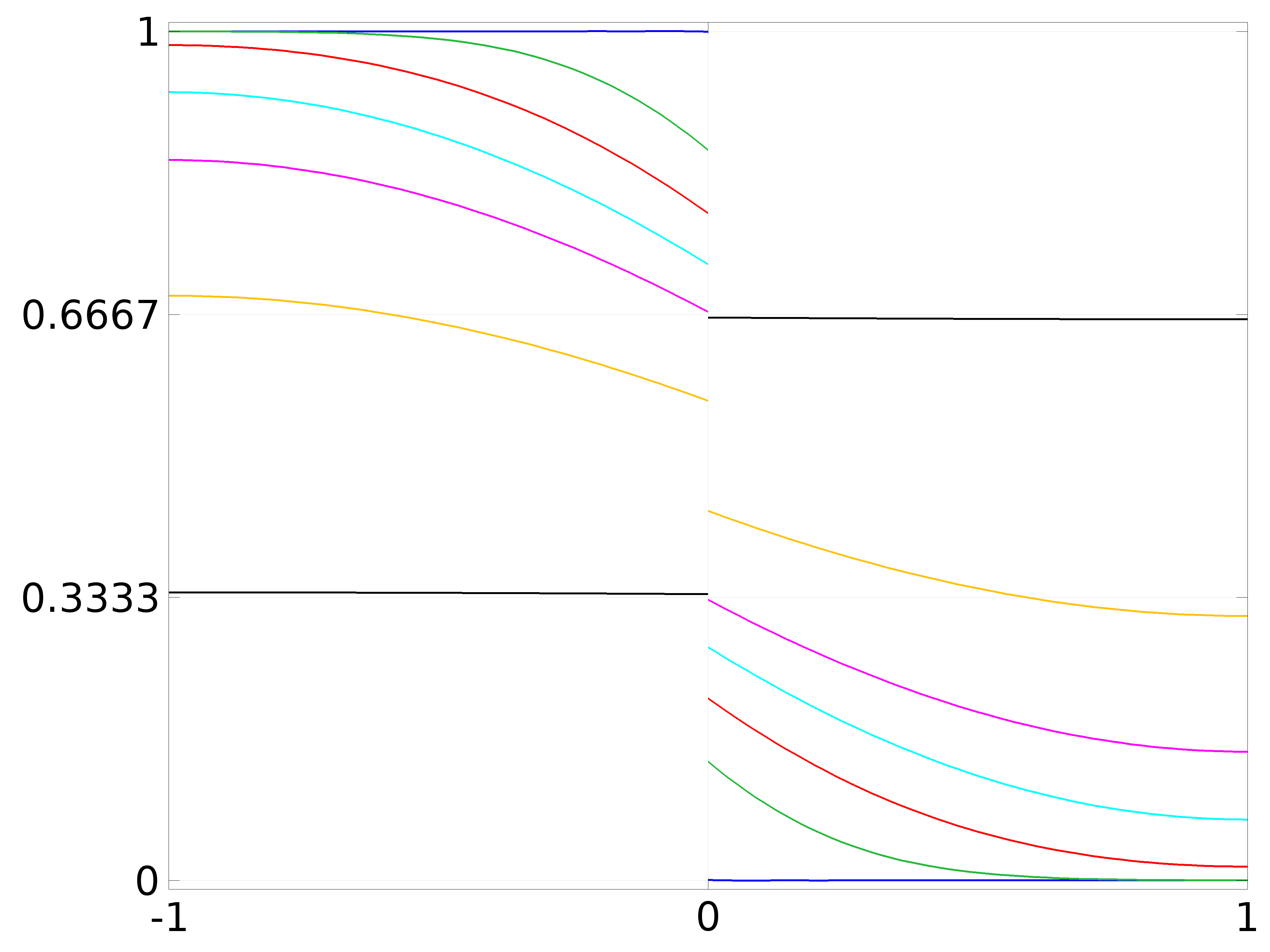} 

\vspace{0.5cm}

\includegraphics[width=0.45 \linewidth]{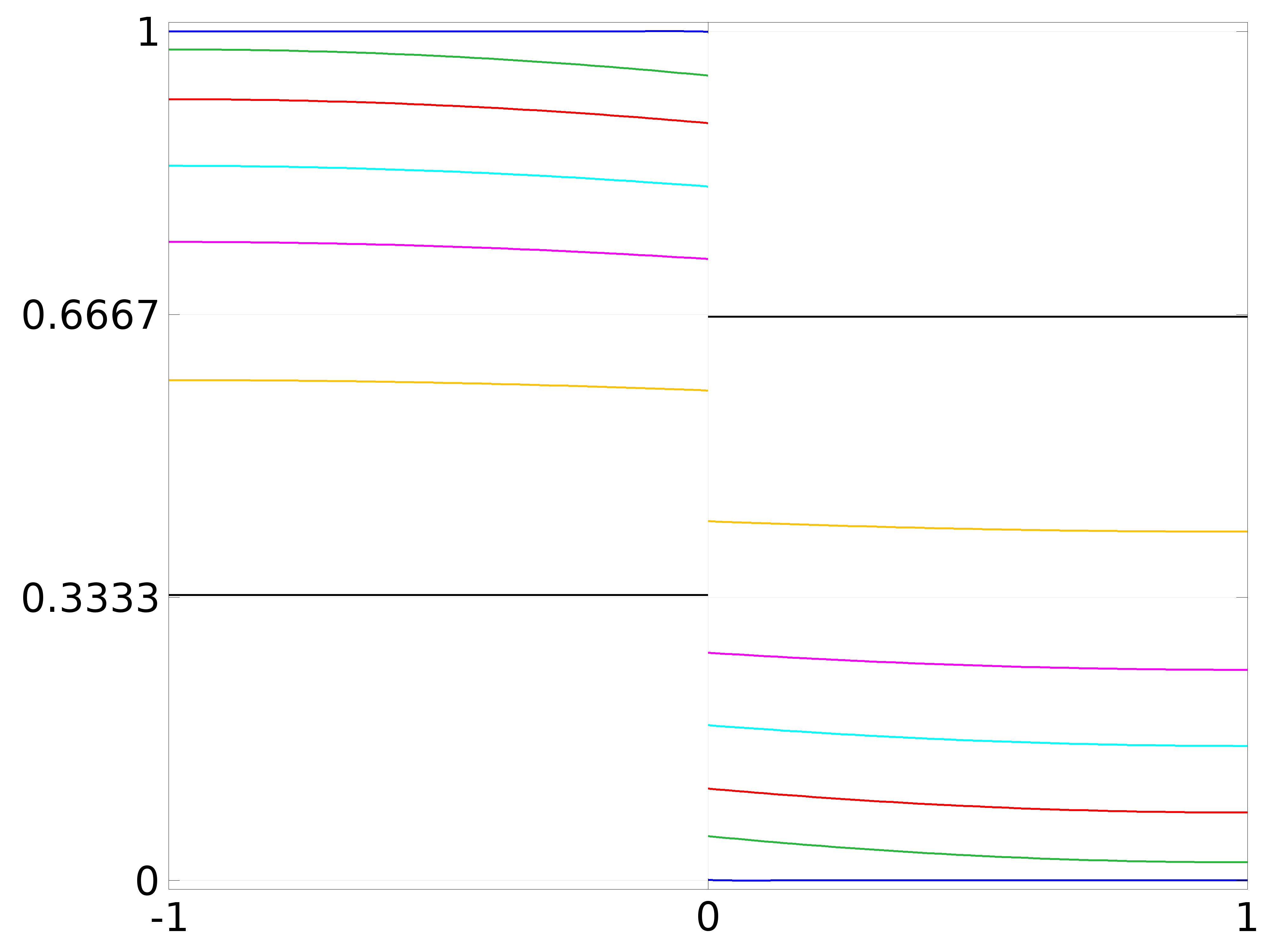} 
\includegraphics[width=0.45 \linewidth , height = 0.33 \linewidth]{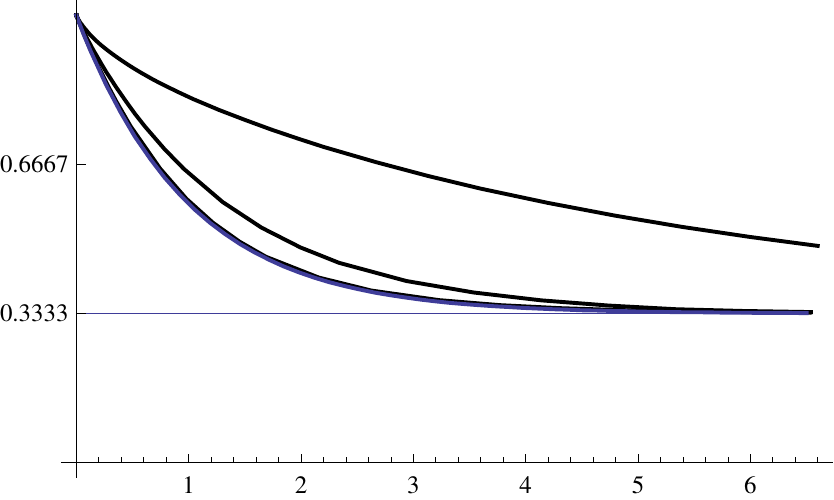} 
\caption{{\footnotesize{
Upper left, upper right and bottom left: Time snapshots of 
 solutions to \eqref{simplenotmodel}--\eqref{transmit} with $a=-1,b=1$,
initial data 
$u_0(x) = 1$ for $x\in (-1,0)$ and $0$ for $x\in (0,1)$,  
 $\alpha=2/3$, $\beta=1/3$, and various values of the diffusion 
coefficient:  $\kappa=0.1$ (upper left),  $\kappa=1$ (upper right) 
and  $\kappa=10$ (bottom left). The curves correspond 
consecutively to times $0, 0.05, 0.15, 0.30, 0.50, 1$ and $6$. ~ 
Bottom right: Time behavior of the function 
$v_-(t)= \int_{-1}^0 u(t,x) \ud x$, i.e., of the probability that a particle diffusing according to the rules \eqref{simplenotmodel}--\eqref{transmit} and starting in the left interval will be there also at time $t$.  The three curves correspond 
to $\kappa=0.1$ (the highest), $\kappa=1$ and $\kappa=10$ (the lowest). 
Within the picture's resolution, the lowest graph coincides with the graph of  
the second coordinate of the solution to \eqref{ukladzik}, i.e., of 
the function $t \mapsto v_{-}(t)$ describing the probability that a particle in the Markov chain related to \eqref{ukladzik}, and starting in the state $-$, will be in that state at $t \ge 0$. 
}}} \label{fig1}
\end{figure}

The main goal of this paper are general theorems describing such limit results (see Theorems \ref{t.l2conv} and \ref{t.lpconv}). In these theorems, intervals are replaced by $d$-dimensional adjacent regions with adequately smooth boundaries (Lipschitz continuity suffices) playing the role of semi-permeable membranes (see Figure \ref{figure1}). Analysis of partial differential equations in dimensions $d
\ge 2$ is technically more demanding than that in $d=1$ (the case just shortly described and discussed in more detail in our previous papers), but the idea is quite the same: We consider  a diffusion process in such adjacent regions, and assume that particles may filter through the membranes from one region to another. As it transpires, if diffusion in each region becomes faster and faster, and at the same time permeability of membranes separating regions diminishes in such a way that fluxes through these membranes remain constant, then the diffusion process so described is well approximated by a Markov chain. 

What happens, figuratively speaking, is that diffusion tries to average everything out, but in practice, ultimately, this averaging takes place in each region separately, since the membranes, being less and less permeable, in the limit become reflecting barriers. Hence, all diffusion can do is to make all points of each region identical, indistinguishable. On the other hand, since the flux through each of the membranes is kept constant along the process of increasing the speed of diffusion, in the limit process there is still some kind of communication, a heat or probability mass exchange between the lumped regions, and interestingly, this communication is that characteristic to a Markov chain. 

\begin{rem} Stochastic analysis allows a deeper insight into the way particles filter through the membrane, under transmission conditions \eqref{transmit}. Each of them, starting in the interval $(0,b)$, performs a Brownian motion in this interval with two reflecting barriers at $x=0$ and $x=b$. However, while times of reflections at $x=b$ are soon forgotten, those at $x=0$ are measured by a highly nontrivial, nondecreasing process, called local time (see e.g. \cite{ito,liggett,revuz}). When an exponential time with parameter $\kappa^{-1} \beta $ with respect to this local time elapses, a particle filters through the membrane, and starts performing a reflected Brownian motion on the other side. (In particular, the larger the $\beta $, the shorter the time `at the membrane' before filtering to its other side, confirming the interpretation of $\beta$ as a permeability coefficient.)  This agrees nicely with the description of the limit process, in which a particle in the right `container' stays there for an exponential time with parameter $\beta b^{-1}$. Hence, from the viewpoint of stochastic analysis, our result says that as $\kappa \to \infty$,  the local time, i.e., the time the process spends at the membrane, divided by $\kappa$, becomes the standard time divided by $b$. \end{rem}

\subsection{An informal introduction to mathematical tools}

A short discussion of mathematical tools used to accomplish our goal is now in order. As we have already mentioned, analysis of partial differential equations in $d \ge 2$ dimensions is significantly more involved than in $d=1$. By far the simplest (of several known) approaches is that via sesquilinear forms. To begin with let us note that equations \eqref{simplemodel} and 
\eqref{simplenotmodel} are particular instances of equations of the form 
\begin{equation}\label{niewiemco}  \frac {\partial u (t,x)}{\partial t} = \mc L u (t,x), 
\end{equation}
where $x$ belongs to a certain region in $\mathbb R^d$, and $\mc L$ is a differential operator. In both examples presented above $\mc L$ is the operator of second derivative; we note, however, that boundary and transmission conditions influence the domain of such an operator in a crucial way. For instance, in the first example this domain is composed of twice continuously differentiable functions whose derivatives at the interval ends vanish. In fact, instead of saying that boundary and transmission conditions influence the domain of $\mc L$, it would be more proper to say that these conditions characterise, and are characterised by this domain. 

It is clear from this short discussion that properties of \eqref{niewiemco} hinge on $\mc L$. In recognition of this fact, and to draw an analogue with the case where $\mc L$ may be identified with a number or a finite matrix, it is customary to write solutions to \eqref{niewiemco} with initial value $u_0$ as 
\[ u(t, \cdot ) = \e^{t\mc L} u_0 ,\]
provided existence, uniqueness and continuous dependence on initial data of such solutions is granted. In such notation, the uniqueness of solutions yields 
\[ \e^{(t+s)\mc L} = \e^{t\mc L} \e^{s\mc L}, \qquad s,t\ge 0\]
which makes the analogue even more appealing, and justifies viewing $\e^{t\mc L}$ as an exponential function of $\mc L.$ We note that, though both $\e^{t\mc L}$ and $\mc L$ are operators (in that they both map functions into functions) their nature is quite different: While $\mc L$ may map functions with small norm into functions of arbitrarily large norms, $\e^{t\mc L}$ cannot do that; there is a universal constant bounding the ratio between the norm of the image and the norm of a non-zero argument of $\e^{t \mc L}.$  Nevertheless, the theory of semigroups of operators (see \cite{goldstein,hp,pazy}, for example) shows that all properties of $\mc L$ have their mirror images in the family $(\e^{t\mc L})_{t
\ge 0} $, and vice versa.

In particular, the initial value problem for \eqref{niewiemco} is well-posed for $u_0$ in the domain of $\mc L$ and solutions depend continuously on initial data $u_0$, if and only if $\mc L$ is a generator of a so-called strongly continuous semigroup of operators, which is nothing else but exponential function of $\mc L$, as introduced above. 

The question of whether a given operator $\mc L$ is a generator is in general quite difficult to answer.  As it turns out, it requires
detailed knowledge of the spectrum of the operator $\mc L$, i.e.\ knowledge of solutions of the elliptic equation \begin{equation}\label{elliptic} \lambda u - \mc L u = f,\end{equation} where $f$ is given and $u$ is searched-for, for a number of complex lambdas. In Hilbert spaces,
these equations can be studied using the theory of sesquilinear forms which takes its origin in the variational approach 
to elliptic partial differential equations. 

To explain, let us look once again at the heat equation with Neumann boundary conditions. For ease of notation we set $\kappa = 1$. Here, we have $\mc L u = u''$, the second derivative, and the domain incorporates the Neumann boundary conditions. 
To solve the elliptic equation $\lambda u - u'' = f$, we take the scalar product with a test function $v\in H^1(0,1)$ (the Sobolev space of functions with weak derivatives in $L^2(0,1)$) and integrate by parts:
\begin{align}
\int_0^1 f(x) \bar v(x)\ud x & = \lambda \int_0^1 u(x) \bar v(x) \ud x- \int_0^1 u''(x) \bar v(x) \ud x\notag\\
& = 
\lambda \int_0^1 u(x)\bar v(x) \ud x + \int_0^1 u'(x) \bar{v}'(x)\ud x, \qquad \forall v \in H^1(0,1),\label{eq.variational}
\end{align}
since the boundary terms vanish due to the Neumann boundary conditions $u'(0)=u'(1) = 0$. Thus, if
$u$ solves our elliptic equation, then $u$ solves equation \eqref{eq.variational}. Doing the above computations in reverse,
we see that, conversely, if $u$ satisfies equation \eqref{eq.variational} for all $v\in H^1(0,1)$ then $u$ satisfies the elliptic equation
$\lambda u - u'' = f$ and the Neumann boundary conditions.

Equation \eqref{eq.variational} is called the variational formulation of the problem. It involves the sesquilinear form
\[
\fq [u,v] \coloneqq \int_0^1 u'(x) \bar v'(x)\ud x
\]
associated to the operator $\cL u = u''$ with Neumann boundary conditions. Here $u$ and $v$ belong to the Sobolev space $H^1(0,1)$.
The relation between the operator $\cL $ and the form $\fq$ is that 
\[ \fq [u,v] = - \langle \cL u , v\rangle, \qquad u \in D(
\mc L), v \in H^1(0,1),\]  
where $\langle \cdot \, , \cdot \rangle$ refers to the usual inner product of $L^2(0,1)$:
\(
\langle u,v \rangle \coloneqq  \int_0^1 u(x)\bar v(x)\ud x.
\)

As it turns out, it is relatively easier to study the form $\fq$ than
the operator $u \mapsto u''$ and  basically all the relevant information is contained in the numerical range 
\[
\Theta (\fq) \coloneqq \{ \fq [u,u] : u \in H^1(0,1) , \|u\|_{L^2(0,1)} =1\}
\]
of the form. The numerical range in turn involves the associated quadratic form $\fq [u] \coloneqq \fq [u,u]$. In the example above, 
we have
\[
\fq [u] = \int_0^1 |u'(x)|^2\ud x,
\]
which can be interpreted as the (kinetic) energy of the `state' $u$. For symmetric forms, i.e.,\ if $\fq [u,v]= \overline {\fq [v,u]}$, the quadratic form can often be interpreted as an energy. 

One should think of quadratic forms as of distant cousins of the quadratic form $\langle Ax, x\rangle = \bar x^\mathsf{T} A x$ associated with a square   
matrix $A$. (If $A$ is $n\times n$, then the form is defined in $\mathbb C^n$.) Note that $A$ is Hermitian, i.e., $A^* = A$, if and only if the quadratic form takes only real values, and that this is a condition on the numerical range. There are other conditions on the numerical range which frequently appear in connection with matrices, for example:  that the numerical range is contained in $(0,\infty)$, resp.\ $(-\infty, 0)$, which are equivalent with the matrix being
positive definite, resp.\ negative definite. The main point is that analogous conditions on forms related to differential operators may be employed in studying solvability of \eqref{elliptic}.

The quadratic forms appearing in the main theorem of this paper are not symmetric, whence the numerical range 
is not contained in the real axis, but it is a subset of the complex plane. As we will see in the next section, the right condition to get
the generator of a semigroup is for a  form to have numerical range in a sector around the real axis.

\subsection{On the adjoint} 
For a form $\fq$ we define its adjoint by $\fq^*[u,v] \coloneqq \overline{ \fq [v,u]} $. In the case considered in this paper, both $\fq$ and $\fq^*$ are related to operators that transpire to be generators of semigroups of operators. These generators, and these semigroups, describe the same stochastic process seen from two different viewpoints, but are related to apparently different transmission conditions; one semigroup provides solutions of the backward Kolmogorov equation, the other provides solutions of the forward Kolmogorov equation. We will explain this point by taking the system \eqref{simplenotmodel}--\eqref{transmit} as a case study. 

\newcommand{\elka}{\mbox{$L^1(a,b)$}}
\newcommand{\elkad}{\mbox{$L^2(a,b)$}}
As already remarked, solutions of the system \eqref{simplemodel}--\eqref{transmit} may be interpreted as densities of temperature distribution. Alternatively, if $u(0,\cdot )$ is a probability distribution density of an initial position of a diffusing particle then $u(t,\cdot)$ is the probability distribution density of this position at time $t$. Hence, the natural space for this system is that of (absolutely) integrable functions on $(a,b)$, denoted $\elka$. In this space the operator $\mc L^*$ governing the whole dynamics is that of the second derivative multiplied by $
\kappa$ (we need to denote this operator by $\mc L^*$ and not by $\mc L$ to comply with the notation in the main body of the text). Importantly, its domain is composed of  members $u$ of \elka \ possessing weak  second derivatives in each of the two subintervals $(a,0)$ and $(0,b)$ that belong to $L^1(a,0)$ and $L^1(0,b)$, respectively, and satisfying transmission and boundary conditions (comp. \eqref{transmit}):
\begin{equation} \label{transmitki}
u '(0+) = u '(0-), \quad \kappa u' (0+)  = \beta u(0+) - \alpha u (0-), \qquad u'(a)=u'(b)=0.
 \end{equation}  

The semigroup generated by $\mc L^*$ operates in \elka , but it turns out that its restriction to the smaller space $\elkad $ is also a semigroup there. (By the way, in the situation considered in this paper and in general quite often it is relatively easier to construct this semigroup in \elkad \, first, and then extend it to \elka . It is the simplicity of the example that allows a direct reasoning -- see also \cite{gregosiewicz}.) 

The form related to $\mc L^*$ (as restricted to \elkad ) is defined on $H^1\subset \elkad$,  the Sobolev-type space of functions having weak derivates on each of the two subintervals $(a,0)$ and $(0,b)$ that belong to $L^2(a,0)$ and $L^2(0,b)$, respectively. Taking a $u$ from the domain of $\mc L^*$ and $v \in H^1$, and integrating by parts we see that
\begin{align*}
-\langle \mc L^* u, v \rangle &= - \kappa \int_{a}^0 u''(x) \overline v(x)\ud x - \kappa \int_0^{b} u''(x) \overline v(x)\ud x  \\& = 
- u'(x) \overline v(x) |_a^0  - u'(x) \overline v(x) |_0^b + \kappa \int_a^b u'(x)\overline v'(x) 
\ud x .  
\end{align*}
Using \eqref{transmitki} we obtain
\( -\langle \mc L^* u, v \rangle  = \fq^* [u,v] \) 
where 
\[ \fq^* [u,v] = (\beta u(0+) - \alpha u(0-)) \overline {(v(0+)- v(0-))} +  \kappa \int_a^b u'(x)\overline v'(x) 
\ud x .   \] 
It follows that 
\[ \fq [u,v] =  {(u(0+)- u(0-))} \overline {(\beta v(0+) - \alpha v(0-))} +  \kappa \int_a^b u'(x)\overline v'(x) 
\ud x .   \] 
Integrating by parts we see that 
\[ \fq [u,v] = -  \langle \kappa u'',v \rangle \]
for any $v \in H^1$ provided that $u$, instead of satisfying \eqref{transmitki}, satisfies 
the following \emph{dual} conditions \begin{equation} \label{transmitkii} u'(0+) = \beta u(0+) - \beta u(0-), \ u'(0-) = \alpha u(0+) - \alpha u(0-), \  u'(a)=u'(b)=0 .\end{equation}
In other words, the operator $\mc L$ related to $\fq $ is also that of the second derivative, but with different domain.

To repeat, even though $\mc L$ and $\mc L^*$ are related to different transmission conditions, the semigroups generated by these operators describe the same stochastic process, say $(X_t)_{t\ge 0}$. While the semigroup generated by $\mc L^*$  describes dynamics of densities of the process, that generated by $\mc L$ describes dynamics of weighted conditional expected values. More precisely, 
for $u \in \elka $, 
\[ \e^{t\mc L} u (x) = E_x u(X_t) \] where $E_x$ denotes expectation conditional on the process starting at $x$. 

In other words, $\mc L^*$ and conditions \eqref{transmitki} are related to the Fokker-Planck equation, known also as Kolmogorov forward equation, while $\mc L$ and conditions  \eqref{transmitkii} are related to the Kolmogorov backward equation for the same process.



\section{Mathematical preliminaries}\label{sec:prelim}
The main tools to prove our averaging principle come from the theory of sectorial forms. In this section we briefly recall the relevant definitions and results. For more information we refer to the books by Kato \cite{kato} and Ouhabaz \cite{ouh05}.

Let $H$ be a Hilbert space. A \emph{sesquilinear form} is a mapping $\fa : D(\fa)\times D(\fa) \to \CC$ which is linear in the first component and antilinear in the second component. Here $D(\fa)$ is a subspace of $H$. If $D(\fa)$ is dense in $H$ we say that $\fa$ is \emph{densely defined}. We write $\fa [u] \coloneqq \fa [u,u]$ for the associated \emph{quadratic form} at $u\in D(\fa)$.

A sesquilinear form $\fa$ is called \emph{sectorial} if the \emph{numerical range}
\[
\Theta (\fa) \coloneqq \big\{ \fa [u] : u \in D(\fa), \|u\|_H \leq 1\big\}
\]
is contained in some sector
\[
\Sigma_\gamma(\theta) \coloneqq \{ z \in \CC : |\arg (z- \gamma)| \leq \theta \}.
\]
Here $\theta \in [0,\frac{\pi}{2})$ is called the \emph{angle} of the sector and $\gamma \in \CR$ is the \emph{vertex} of the sector.
To emphasize $\theta$ we will say $\fa$ \emph{is sectorial of angle $\theta$}.
If $\Re \fa [u] \geq 0$ for all $u\in D(\fa)$ then $\fa$ is called \emph{accretive}. The numerical range of an accretive, sectorial form is always contained in a sector with vertex $0$. We will mainly be interested in accretive forms. 
If the numerical range of a form is contained in $\Sigma_\gamma (0)$, the form is called \emph{symmetric}.

The \emph{adjoint form} of $\fa$ is defined as $\fa^* : D(\fa)\times D(\fa) \to \CC$, $\fa^*[u,v] \coloneqq \overline{ \fa [v,u]}$. If
$\fa$ is sectorial (accretive, symmetric) then so is $\fa^*$. The \emph{real} and \emph{imaginary parts} of $\fa$
are defined by $\Re \fa \coloneqq \frac{1}{2}(\fa + \fa^*)$ and $\Im \fa \coloneqq \frac{1}{2i} (\fa - \fa^*)$, respectively. An easy computation shows that $\Re\fa$ and $\Im\fa$ are symmetric forms and that $\fa[u,v] \coloneqq \Re\fa [u,v] + i\Im \fa [u,v]$.
We point out that even if the quadratic forms associated with $\Re\fa$ and $\Im \fa$ take only values in the real numbers,
the forms themselves may take values in the complex plane. It is easy to see that an accretive sesquilinear form is sectorial of angle 
$\theta$ if and only if
\[
|\Im \fa [u] | \leq \tan\theta \Re\fa [u].
\]
In this case
\[
\la u,v\ra_{\fa} \coloneqq \Re\fa[u,v] + \la u,v\ra_H
\]
defines an inner product on $D(\fa)$. This is called the \emph{associated inner product}. If this inner product turns $D(\fa)$ into a Hilbert space, $\fa$ is called \emph{closed}.\medskip
  
If $\fa$ is a densely defined sesquilinear form, we define the \emph{associated operator} $\cL$ by setting
\[
D(\cL) \coloneqq \{ u \in D(\fa) : \exists w \in H \mbox{ with } \fa [u,v] = \la w,v\ra \mbox{ for all } v \in D(\fa)\},
\quad \cL u = - w.
\]
Note that by the density of $D(\fa)$ in $H$ there exists at most one such $w$. We now have the following result, see \cite[Theorem VI.2.1]{kato} or \cite[Section 1.4]{ouh05}

\begin{thm}\label{t.assop}
Let $\fa$ be an accretive, closed and densely defined 
sectorial form of angle $\theta$. 
Then the associated operator $\cL$ is closed
and densely defined, $\CC \setminus \overline{\Sigma_0(\theta)}$ is contained in 
the resolvent set $\rho (\cL)$ and
\[
\|R(\lambda, \cL ) \| \leq \frac{1}{\mathrm{dist}(\lambda, \Sigma_0(\theta))} \qquad
\mbox{for all } \lambda\in \CC \setminus \overline{\Sigma_0(\theta)}.
\]
In particular, $\cL$ is sectorial of angle $\frac{\pi}{2} - \theta$ in the sense of \cite[Definition II.4.1]{en}
and thus generates a bounded, analytic semigroup of operators in $H$ by \cite[Theorem II.4.6]{en}. This semigroup is contractive
on $[0,\infty)$.
\end{thm}

If the form $\fa$ is closed and sectorial but not densely defined, there is no associated operator in $H$. However, we may associate an operator
$\cL|_{H_{0}}$ on the Hilbert space $H_{0} \coloneqq \overline{D(\fa)}^H$. We will also in this situation call
$\cL|_{H_{0}}$ the operator associated with $\fa$. As a consequence of Theorem \ref{t.assop}, 
$\cL_{H_{0}}$ is a sectorial operator on $H_{0}$ and thus generates a bounded, analytic semigroup $e^{t\cL_{H_{0}}}$ on 
$H_{0}$. Following Simon \cite{sim78}, who treated the symmetric case, we extend each operator of the semigroup to $H$ by setting it to 
$0$ on $H_{0}^\perp$. In other words, for $u= u_0+u_1 \in H_{0} \oplus H_{0}^\perp = H$ we define the \emph{(degenerate) semigroup}
$e^{-z\fa}$ by setting
\[
e^{-z\fa}u \coloneqq e^{z\cL_{H_{0}}}u_0
\]
and the pseudoresolvent $(\lambda + \fa)^{-1}$ by setting
\[
(\lambda + \fa)^{-1}u \coloneqq R(\lambda, \cL_{H_{0}}) u_0.
\]
{With slight abuse of notation, we write
\begin{equation}\label{eza}
e^{-z\fa} = e^{z\cL_{H_{\fa}}}P_{H_{0}} \quad\mbox{and}\quad (\lambda + \fa)^{-1} = R(\lambda, \cL_{H_{0}})P_{H_{0}}
\end{equation}
where $P_{H_{0}}$ is the orthogonal projection onto $H_{0}$.}
Note that the relationship between $e^{-z\fa}$ and $(\lambda+\fa)^{-1}$ is the same as in the case where $\fa$ is densely defined. Namely, $e^{-z\fa}$
can be computed from $(\lambda+\fa)^{-1}$ via an appropriate contour integral and, vice versa, $(\lambda+\fa)^{-1}$ is the Laplace transform
of $e^{-z\fa}$. 

As in \cite{sim78} the main motivation to consider forms which are not densely defined are convergence results for
sectorial forms, where non densely defined forms may appear naturally in the limit, even if we consider a sequence of densely defined 
sectorial forms. In fact, this is exactly what happens in our averaging principle. To prove it, we make use of the following 
convergence result due to Ouhabaz \cite{ouh95} which generalizes Simon's theorem \cite{sim78} concerned with symmetric forms.
See also the recent article \cite{bte14} for related results. 

\begin{thm}\label{t.ouhabaz}
Let $\fa_n, n \ge 1$ be a sequence of accretive, closed and uniformly sectorial forms in a Hilbert space $H$. The latter means that 
all the numerical ranges are contained in a common sector $\Sigma_0(\theta)$.
Moreover, assume that
\begin{enumerate}
[(a)]
\item $\Re\fa_n \leq \Re\fa_{n+1}$, i.e.,\ $D(\fa_{n+1}) \subset D(\fa_n)$ and
$\Re \fa_n[u] \leq 
\Re \fa_{n+1}[u]$ for all $u\in D(\fa_{n+1})$;
\item we either have
\begin{enumerate}
[(i)]
\item $\Im \fa_n[u] \leq \Im\fa_{n+1}[u]$ for all $u\in D(\fa_{n+1})$ or
\item $\Im\fa_{n+1}[u] \leq \Im \fa_n[u]$ for all $u\in D(\fa_{n+1})$.
\end{enumerate}
\end{enumerate}
Define $\fa[u]\coloneqq \lim_{n\to\infty}\fa_n[u]$ with domain
\[
D(\fa) \coloneqq \Big\{ u \in \bigcap_{n\in \CN} : \sup_{n\in\CN} \fa_n[u] < \infty\Big\}.
\]
Then $\fa$ is an accretive, closed and sectorial form, and $\fa_n$ converges to $\fa$ in the strong resolvent sense, i.e.,\
\[
(\lambda+ \fa_n)^{-1}u \to (\lambda +\fa)^{-1}u \quad\mbox{as } n\to\infty,
\]
for all $u\in H$ and $\lambda \in \CC\setminus \overline{\Sigma_0(\theta)}$.
\end{thm}

Ouhabaz has proved this theorem only for densely defined forms but inspection of the proof shows that it generalizes also to non densely defined forms. Indeed, {besides properties of analytic functions, the proof only makes use of Simon's monotone convergence theorem \cite{sim78} which is valid also for non densely defined forms.}

An important consequence of Theorem \ref{t.ouhabaz} is the following.

\begin{cor}\label{c.semigroupconv}
In the situation of Theorem \ref{t.ouhabaz} we have $e^{-t\fa_n}u \to e^{-t\fa}u$ as $n\to \infty$ for all $u\in H$ and $t\geq 0$.
\end{cor}

\begin{proof}
To see this note that the degenerate semigroup $e^{-t\fa_n}u$ can be computed from the pseudoresolvent
$(\lambda+ \fa_n)^{-1}u$ via a contour integral. By the strong resolvent convergence the integrands converge pointwise on the contour
to $(\lambda +\fa)^{-1} u$. However, as our forms are uniformly sectorial, it follows from Theorem \ref{t.assop} that the associated operators
are uniformly sectorial. From this, we obtain an integrable majorant for $(\lambda+ \fa_n)^{-1}u$. The thesis now follows from the dominated convergence theorem. See e.g. \cite[Proposition 4]{deg} for details. 
\end{proof}

{The situation where degenerate semigroups (or, equivalently, non densely defined operators) appear in convergence results is quite common in applications and can be studied in more generality in the framework of singular perturbation problems, see 
 \cite{banalachokniga} and \cite{knigazcup}.  We would like to point out that the situation in Corollary \ref{c.semigroupconv}
is rather special in that in general mere convergence of the resolvents does not imply convergence of the related semigroups       
 (see examples in e.g. \cite{deg} or \cite[Chapter 8]{kniga}). What allows us to infer convergence of the semigroups from that of the resolvents is the fact that the related semigroups are uniformly holomorphic. That this is helpful in convergence results has been known for a time (see, e.g., \cite{a01,deg} and the seminal paper \cite{dapratos}).  
 
In the general situation where the semigroups considered are not uniformly holomorphic, more refined techniques are needed to establish convergence of the semigroups. It is worth noticing that many singular perturbation problems  which not necessarily involve uniformly holomorphic semigroups, fall into an ingenious scheme devised by T. G. Kurtz \cite[pp.\ 39-42]{ethier}\cite{kurtzper,kurtzapp}. In fact, relatives of our averaging principle can also be deduced from Kurtz's theorem, see
\cite[Chapter 42]{knigazcup}; the same is essentially true of the Freidlin--Wentzell principle \cite{emergence}.  
}

\section{Notation and Assumptions}
\subsection{Domains and their boundaries}
We let $\Omega_0 \subset \CR^d$ be a connected, bounded open domain with Lipschitz 
boundary. Here, we say that an open set $\Omega \subset \CR^d$ has \emph{Lipschitz boundary} if it is locally the epigraph of a Lipschitz function, see  e.g. \cite[p.\ 111] {agranovicz}. 
More precisely, given $z\in \partial \Omega$ we may find an open neighborhood $V$ of $z$ in $\partial \Omega$ such that there is (a) a cylinder $C=B \times (a,b)$, where $B$ is an open ball in $\CR^{d-1} $ and $(a,b)$  is an open subinterval of $\CR$, and (b) an isomorphism $\mcj $ of $\CR^d$, and (c) a Lipschitz continuous functions $g: B \to \CR $ such that defining $\phi (w,t) = t- g (w)$ for $(w,t) \in C$, we have 
$\Omega \cap C = \mcj\{\phi < 0\}, C \setminus \overline{\Omega} = \mcj \{\phi > 0\},$  and $V = \mcj \{\phi =0\}$. Our domain $\Omega_0$ is further partitioned (see Figure \ref{figure1}), i.e., we consider subsets $\Omega_1, \ldots, 
\Omega_{N-1} \subset \Omega_0$ that are pairwise disjoint and open with Lipschitz boundary. 
We set
\[
\Omega_N \coloneqq \interior \big( \Omega_0 \setminus \bigcup_{k=1}^{N-1}\overline{\Omega}_k \big).
\]
We assume that also $\Omega_N$ is a bounded open set with Lipschitz boundary. This assumption excludes certain configurations of the sets $\Omega_1, \ldots, \Omega_{N-1}$. For example, it may not happen that we have two balls which touch in exactly one point. 
On the other hand, it is no restriction to assume that the sets
$\Omega_1, \ldots, \Omega_N$ are connected, otherwise we consider the connected components of these sets. 
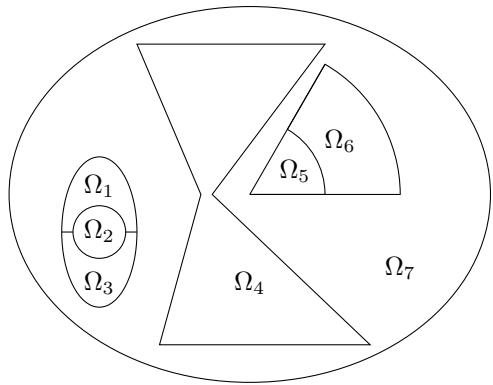
\begin{figure}
\begin{tikzpicture}
   \draw (0:0cm) -- (0:2cm)
         arc (0:60:2cm) -- (60:1cm)
         arc (60:0:1cm) -- (60:0cm)
-- cycle;\draw (0,0) -- (1,1.732);
\draw (0,0) ellipse (3.2cm and 2.5cm);
\draw (-2,-0.5) ellipse (0.5cm and 1 cm)
               ellipse (0.35cm and 0.35cm);
\draw (-0.65,0) -- (-1.2,-2) -- (1.6,-2) -- (-0.5,0) -- (1,2)-- (-1.5,2)--cycle;
\draw (-1.65,-0.5)--(-1.5,-0.5);
\draw (-2.5,-0.5)--(-2.35,-0.5);
\node [above] at (-2,-0.1) {$\Omega_1$};
\node [above] at (-2,-0.7) {$\Omega_2$};
\node [above] at (-2,-1.4) {$\Omega_3$};
\node [above] at (-0,-1.4) {$\Omega_4$};
\node [above] at (2,-1.2) {$\Omega_7$};
\node [above] at (0.6,0.1) {$\Omega_5$};
\node [above] at (1.2,0.5) {$\Omega_6$};
\end{tikzpicture} \caption{Partition of $\Omega_0$ (here, an ellipse in $\CR^2$) into $N=7$ parts.}\label{figure1}\end{figure}

The boundary of the set $\Omega_k$ is denoted by $\Gamma_k$ for $k=0, \ldots, N$. We write 
\[
\Gamma_{k,\ell} \coloneqq \Gamma_k \cap \Gamma_\ell
\]
for the common boundary of $\Omega_k$ and $\Omega_\ell$ ($k,\ell = 0, \ldots, N, k \not = \ell$). To simplify some formulas to be discussed later, we also agree that \[\Gamma_{k,k} = \emptyset \] for $k =1,\ldots, N$. Note that we may also well have 
that $\Gamma_{k,\ell} = \emptyset$ for certain values of $k \not = \ell$. Below, we always endow the boundaries
$\Gamma_k$ with their natural surface measure $\sigma_k$. {Actually} $\sigma_k$ coincides 
with $(d-1)$-dimensional Hausdorff  measure $\mathscr{H}^{d-1}$ 
since all appearing domains have a Lipschitz boundary. As there is no chance of confusion, 
we drop the index $k$ and write $\sigma$ for the surface measure on any of the $\Gamma_k$.

Below, we make use of the following observation; its proof is relegated to Appendix. 

\begin{lem}\label{l.intersect}
If $k, \ell, m$ are distinct numbers between $0$ and $N$, then the $(d-1)$-dimensional Hausdorff measure of 
$\Gamma_k\cap \Gamma_\ell\cap \Gamma_m$ is zero.
\end{lem}

\subsection{Diffusion and permeability of membranes}\label{dpm} 
Let \begin{equation}\label{en} \mc N = \{1,\dots,N\} \quad \text{ and } \quad \mc N_0 = \{0,1,\dots,N\}.\end{equation}
It is our aim to study diffusion on $\Omega_0$ with the sets $\Gamma_k$ (for $k\in 
\mc N$)
modeling semi-permeable membranes (see below). As far as our diffusion coefficients $A = (a_{ij}) \in L^\infty(\Omega_0; \CR^{d\times d})$ are concerned, we make 
the following assumptions.
\begin{enumerate}
[(i)]
\item They are symmetric, i.e.,\ $a_{ij}= a_{ji}$ for $i,j= 1, \ldots, d$.
\item They are uniformly elliptic, i.e.,\ there exists a constant $\gamma>0$ such that for any vector $\xi \in \CC^d$ we have
\[
\sum_{i,j=1}^d a_{ij}(x) \xi_i\bar\xi_j \geq \gamma \|\xi\|^{2} = \gamma\sum_{j=1}^d |\xi_j|^2
\]
for almost every $x\in \Omega_0$.
\end{enumerate}

The differential operator we are interested in is formally given by
\begin{equation}\label{operatorl}
\mathscr{L} u = \sum_{i,j=1}^d \partial_i (a_{ij}\partial_j u) {-cu}= \div (A\nabla u){-cu},
\end{equation}
where $c \in L^\infty(\Omega_0)$ is a given non-negative function playing the role of a potential.  
To define a suitable realization of $\cL$ in $L^2(\Omega_0)$ we use form methods. The related form is defined on the space $\ha \subset L^2(\Omega_0)$:
\[
\ha 
\coloneqq \{ u\in L^2(\Omega_0) : u|_{\Omega_k} \in H^1(\Omega_k)\,\,
\forall \, k=1, \ldots, N\}.
\]
Obviously, $\ha$ 
is a Hilbert space with respect to the inner product
\[
\langle u, v \rangle_{\ha} \coloneqq \int_{\Omega_0} u {\bar v}\dlam + \sum_{k=1}^N \int_{\Omega_k} \nabla u\cdot{\overline{\nabla v}}\dlam ,
\]
where $\lam $ denotes the $d$-dimensional Lebesgue measure, and $\overline{z}$ is the conjugate of a complex number $z$.  
 
For $u \in \ha$ 
the function
$u|_{\Omega_k}$ has a trace in $L^2(\Gamma_k)$ as $\Omega_k$ has a Lipschitz boundary. We denote the trace of $u|_{\Omega_k}$
by $u_{|k}$. Note that we can have $u_{|k}\neq u_{|\ell}$ on $\Gamma_{k,\ell}$. Thus, 
we should interpret $u_{|k}$ as `the values of $u$ on the boundary $\Gamma_{k,\ell}$ when approached from within $\Omega_k$', whereas $u_{|\ell}$
are `the values on the boundary when approached from within $\Omega_\ell$'.

We imagine that a diffusing particle in any subdomain $\Omega_k$ may permeate through a semi-permeable membrane, i.e., {through} the boundary $\Gamma_{k,\ell}$ separating this subdomain from the neighboring subdomain $\Omega_\ell$. The membrane's permeability may change along the boundary. In particular, the permeability may vary from $\Gamma_{k,\ell}$ to $\Gamma_{k,\ell'}$. This is modeled by permeability functions $\tau_k$ defined on  $\Gamma_k $; we assume $\tau_k \in L^\infty(\Gamma_k; \CR)$ with $\tau_k\geq 0$
almost everywhere. By analogy with the analysis in Section \ref{s.gi}, the value of $\tau_k$ at a point $x$ of the boundary should be thought of as a permeability coefficient of the membrane at this point. Roughly speaking, 
the larger $\tau_k (x)$, the less time it takes on average to permeate through the membrane at $x$, when approaching from within $\Omega_k$ (see also the discussion after \eqref{eq.trans}, further on). {Note that by Lemma \ref{l.intersect}, up to a set of measure zero,
there is only one adjacent set $\Omega_\ell$ to which the particle may permeate.}

Moreover, we are given measurable functions $b_{k,\ell} : \Gamma_{k,\ell} \to [0,1]$ for $1\leq k,\ell\leq N$. The quantity $b_{k,\ell} (x) $ describes the possibility that a particle right after filtering from $\Omega_k$ through the membrane $\Gamma_{k,\ell}$ at a point $x$, instead of starting diffusion in $\Omega_\ell$, will be immediately killed and removed from the state-space. For $b_{k,\ell} (x) =1$ all particles survive filtering through the membrane at this point, for $b_{k,\ell} (x) =0$ none of them does.

To formulate our transmission conditions, we need to define the conormal derivative associated
with $\mathscr{L}$ on the domain $\Omega_k$. We do this in a variational sense.

\begin{defn}
Let $u\in \ha$ 
be such that $\div (A \nabla u)|_{\Omega_k} \in L^2(\Omega_k)$. Then 
there exists a unique function $N_k(u)\in L^2(\Gamma_k,\sigma)$ such that
\[
\int_{\Omega_k} \div (A \nabla u) {\bar v}\dlam + \int_{\Omega_k} (A\nabla u)\cdot {\overline{\nabla v}}\dlam = 
\int_{\Gamma_k} N_k(u) {\bar v_{|_k}}\, \ud \sigma
\]
for all $v \in H^1(\Omega_k)$. We call $N_k(u)$ the \emph{conormal derivative of $u$ on $\Omega_k$}.
\end{defn}

To see that such a function $N_k(u)$ exists, let $\Phi : L^2(\Gamma_k) \to H^1(\Omega_k)$ be a continuous linear mapping such
that $\Phi (g)|_k = g$. We can for example pick $\Phi (g)$ as the unique solution of the Dirichlet problem
\[
\left\{ \begin{array}{rcl}
\div (A\nabla u) & = & 0\\
u|_{\partial\Omega_k} &=& g.
\end{array}
\right.
\]
If $u\in H^1(\Omega_k)$ is such that $\div (A\nabla u) \in L^2(\Omega_k)$ we can consider the map $\varphi_u : L^2(\Gamma_k) \to \CC$
defined by
\[
\varphi_u(g) \coloneqq \int_{\Omega_k} \div (A\nabla u) \overline{\Phi (g)}\dlam + \int_{\Omega_k} 
(A\nabla u) \overline{\nabla \Phi (g)}\, \dlam.
\]
Since $\Phi$ is continuous $\varphi_u$ is a continuous, antilinear functional on $L^2(\Gamma_k)$. Hence it follows from the theorem 
of Riesz--Fr\'echet that there exists a unique element $N_k(u)$ such that
\[
\varphi_u(v) = \int_{\Gamma_k} N_k(u)\bar v\, \ud\sigma.
\]

In the situation where everything is smooth, {it follows from the divergence theorem that}
\[
N_k(u) = \sum_{i,j=1}^d a_{ij}\partial_j u\nu_j
\]
where $\nu=(\nu_1, \ldots, \nu_d)$ is the outer normal to $\Omega_k$.

{Our transmission conditions are:
\begin{equation}\label{eq.trans}
N_k(u) = -\tau_k (u_{|k}-b_{k,\ell}u_{|\ell}) \quad\mbox{ on } \Gamma_{k,\ell}
\qquad \text { for all $k,\ell \in \mc N$};
\end{equation} 
comp. \eqref{transmitkii} (their dual form is presented in \eqref{newtc}, further down).}
Probabilistically, these conditions may be interpreted as follows: a particle diffusing in a region $\Omega_k$ `bounces' from the membrane separating it from $\Omega_\ell $, similarly to the reflected Brownian motion, but the time it spends `at the membrane' is measured, and after an exponential time with respect to this reference measure elapses, the particle filters through to $\Omega_\ell $. The larger the $\tau_k$ at an infinitely small part of the membrane the larger the parameter in the exponential time, and the shorter the time it takes to filter through that part of the membrane. Additionally, as described above, functions $b_{k,\ell}$ describe the possibility that a particle will be killed after filtering through the membrane $\Gamma_{k,\ell}$. 

For each $k \in \mc N$, on the part $\Gamma_{k,0}$ of the outer boundary $\Gamma_0$, we impose the Robin boundary conditions
\begin{equation}\label{eq.robin}
N_k(u)  = -\tau_ku_{|k},  
\end{equation}
and note that this is reduces to \eqref{eq.trans} for $\ell=0$ when agreeing
\[b_{k,0} =0 ,\] 
i.e.,\ that all particles filtering from $\Omega_0$ to its complement are immediately killed, and removed from the state-space.
 

\subsection{The related quadratic form}\label{rqf}
Let us assume that $u\in \ha $ 
is such that
$\div (A\nabla u) \in L^2(\Omega_0)$ and such that the transmission conditions \eqref{eq.trans} and the boundary condition \eqref{eq.robin} are satisfied. Then for a function $v \in \ha$ 
we have
\begin{align*}
 - \int_{\Omega_0}\div (A\nabla u)\bar v \dlam & =  
   -\sum_{k\in \mc N}\int_{\Omega_k} \div (A\nabla u)\bar v\dlam\\
& = \sum_{k\in \mc N} \Big( \int_{\Omega_k} (A\nabla u)\cdot \overline{\nabla v}\dlam - \int_{\Gamma_k} N_k(u) \bar v_{|k}\, \ud \sigma\Big)\\
& = \int_{\Omega_0} (A\nabla u)\cdot \overline{\nabla v}\dlam +\sum_{\substack{ k \in {\mc N}\\ \ell \in {\mc N}_0}}
\int_{\Gamma_{k,l}} \tau_k (u_{|k}-{b_{k,\ell}}u_{|\ell}) \bar v_{|k}\, \ud \sigma. 
\end{align*}
Here the first equality uses the fact that the  $d$-dimensional Lebesgue measure of the set $\Omega_0 \setminus \bigcup_{k=1}^N\Omega_k$ is zero,
the second equality is the {definition of the conormal derivative (or the divergence theorem in the smooth case)}, the third one follows from Lemma \ref{l.intersect}, the transmission conditions \eqref{eq.trans} and the boundary condition \eqref{eq.robin}.

This calculation leads us to define the following forms on the Hilbert space $L^2(\Omega_0)$.
\begin{defn}\label{def.forms}
For a parameter $\kappa \ge 0 $, we define the form $\fq_\kappa$ by setting
\[
\fq_\kappa [u, v] \coloneqq \int_{\Omega_0} \left ( \kappa (A\nabla u)\cdot \overline{\nabla v}+ c u \bar v\right ) \,\dlam,
\]
and the form $\fb$  by
\[
\fb [u,v] \coloneqq \sum_{\substack{k \in \mc N\\\ell\in \mc N_0}} \int_{\Gamma_{k,l}} \tau_k (u_{|k}-{b_{k,\ell}}u_{|\ell}) \bar v_{|k}\, \ud \sigma
\]
for $u,v$ in the common domain $ D(\fb)= D(\fq) \coloneqq \ha $.  We put
$\fa_\kappa \coloneqq \fq_\kappa +\fb$. 
The adjoint $\fb^*$ of $\fb$ is given by
\begin{equation}\label{eq.adjoint}
\fb^* [u,v] = \sum_{\substack{k \in \mc N\\\ell\in \mc N_0}} \int_{\Gamma_{k,l}} \tau_k u_{|k}(\bar v_{|k}-{b_{k,\ell}}\bar v_{|\ell}) \, \ud \sigma.
\end{equation}
Since $\fq$ is symmetric we have $\fa_\kappa^* = \fq_\kappa+\fb^*$.
\end{defn}

\begin{rem}\label{dualnosc}
Rearranging the terms in \eqref{eq.adjoint}, we see that
\[
\fb^*[u,v]= \sum_{\substack{k \in \mc N\\\ell\in \mc N_0}} \int_{\Gamma_{k,l}} (\tau_k u_{|k} - \tau_{\ell} b_{k,\ell}u_{|\ell}) \bar{v}_{|k}\, \ud\sigma.
\]
Repeating the computations from the beginning of this subsection, we conclude that functions in the domain of the adjoint operator
satisfy the following transmission conditions {(compare \eqref{transmit} in our introductory  Section \ref{s.gi}):
\begin{equation}\label{newtc}
N_k(u) = - (\tau_k u_{|k} - \tau_\ell b_{k,\ell} u_{|\ell}) \quad \mbox{on } \Gamma_{k,\ell}\quad\quad \mbox{for all } k, \ell \in \mc N.
\end{equation}
To repeat, both forms describe the same stochastic process. } More specifically, if $S$ is the semigroup to be introduced in the next section, and `generated' by $\fa_\kappa^*$, and if a non-negative $u \in L^2(\Omega)$ is the initial distribution of the 
underlying stochastic process $(X_t)_{t\ge 0} $ in $\Omega_0$, then $S(t) u$ is the distribution of this process at time $t\ge0.$ On the other hand, the semigroup $T$ of the next section, `generated' by $\fa_\kappa$, speaks of the dynamics of expected values: for $u \in L^2(\Omega_0)$,  $T(t) u (x) $ is the expected value of $u (X_t)$ conditional on $X_0 = x, x \in \Omega_0.$ The semigroup $S$ provides solutions of the Fokker-Planck equation, or the Kolmogorov forward equation, while $T$ provides solutions of the Kolmogorov backward equation.

As we shall see in examples of Section \ref{sec:6}, both \eqref{eq.trans} and \eqref{newtc} are used in practice, 
depending on what is the quantity modeled.
  
\end{rem}

\section{Generation Results}\label{sect.gen}

{In this section, we prove that for $\kappa \ge 1$ the operator associated with $\fa_\kappa$ generates a strongly continuous, 
analytic and contractive semigroup on $L^2(\Omega_0)$.  To simplify notation, we write $\fa = \fa_1$ and $\fq = \fq_1$ and 
prove the main results for this form. Note, however, that they also apply to $\fa_\kappa$ for $\kappa \geq 1$, as is seen by changing the diffusion coefficients matrix $A$.} The following proposition is the key step towards a generation result on $L^2(\Omega_0)$.

\begin{prop}\label{p.sectorial}
The forms $\fa$ and $\fa^*$ are closed, accretive and sectorial. 
\end{prop}

\begin{proof}
Since $A$ is symmetric and $c$ is real valued with $c\geq 0$, 
the form $\fq$ is symmetric and obviously we have $\fq[u] \geq 0$ for all $u\in \ha $. 
Let $M \coloneqq \sup_{x\in \Omega_0} \|A(x)\| {< \infty}$. We then have
\[
\gamma |\nabla u|^2 \leq \la A\nabla u, \nabla u\ra \leq |A\nabla u||\nabla u| \leq M|\nabla u|^2
\]
almost everywhere. Integrating this  inequality over $\Omega_0$  and adding a suitable multiple of 
$\|u\|_{L^2(\Omega_0)}^2$ we see that
\[
\min\{1,\gamma\} \|u\|_{\ha}^2 \leq \fq[u] +\|u\|_{L^2(\Omega_0)}^2 \leq \max\{1 + \|c\|_{L^\infty(\Omega_0)} , M\} \|u\|_{\ha}^2.
\]
Thus, the inner product
$\langle u, v\rangle_{\fq} \coloneqq \fq[u,v]  + \langle u, v\rangle_{L^2(\Omega)}$ is equivalent to the canonical inner product in $\ha$, i.e., the related norms are equivalent.
This yields the closedness of $\fq$. 

To prove that $\fa$ is closed and sectorial we show that $\fb$ is $\fq$-bounded with $\fq$-bound $0$, i.e.,\ for every $\eps>0$
there exists a constant $C(\eps)$ such that 
\[
|\fb [u]| \leq \eps \fq [u] + C(\eps)\|u\|_{L^2(\Omega)}.
\]
For the proof, let a sequence $u_n$ be given with $u_n \weak 0$ in $\ha$. 
It follows from the 
compactness of the trace operator (which is a consequence of the Lipschitz nature of the boundary, see \cite[Theorem 2.6.2]{nevas}) that we have $u_{n|k} \to 0$ in 
$L^2(\Gamma_k,\sigma)$ for $k=0, \ldots, N$. As the functions $\tau_0, \ldots, \tau_n$ {and $b_{k,\ell}$ ($1\leq k,\ell \leq N$) }
are bounded, it follows from the 
Cauchy--Schwarz inequality that $\fb[u_n] \to 0$. Since $\fb$ is bounded on $D(\fq)$ (a consequence of the boundedness of
the trace operator) it now follows from \cite[Lemma 7.4]{dan13} that $\fb$ is $\fq$-bounded with $\fq$-bound $0$.

A perturbation result for sectorial forms \cite[Theorem VI.1.33]{kato} yields that $\fa$ is a closed and sectorial form; moreover, the associated inner product is equivalent to that associated to $\fq$ and thus it is equivalent to the canonical inner product in 
$\ha $, {proving that $\fa$ is closed}.
Finally, as the $\fq$-bound of $\fb$ is $0$, it follows that $\fa$ is accretive as well.

With the same reasoning we can show that $\fa^*$ is closed, sectorial and accretive.
\end{proof}

We denote by $\cL$ the associated operator of $\fa$ in $L^2(\Omega_0)$. From Theorem \ref{t.assop} we obtain the following result.

\begin{cor}\label{c.l2gen}
The operator $\cL$ generates a strongly continuous, holomorphic and contractive 
semigroup $T=(T(t))_{t\geq 0}$ on $L^2(\Omega)$. The operator $\cL^*$ generates a strongly continuous, holomorphic and contractive semigroup $S = (S(t))_{t\geq 0}$ on $L^2(\Omega_0)$. We have $S(t)=T(t)^*$ for all $t\geq 0$.
\end{cor}

\section{Convergence results for fast diffusion}\label{sect.conv}

In this section, we `speed up diffusion' by considering the forms $\fa_\kappa$ with index $\kappa \ge 1$ again. Formally, 
this corresponds to replacing the diffusion matrix $A$ with $\kappa A$. Applying the results of the previous section
to $\fa_\kappa$ and $\fa_\kappa^*$, we obtain semigroups $T_{\kappa} $ and $S_{\kappa}$. Let us denote their generators by $\mc L_{\kappa}$ and $\mc L_{\kappa}^*$ so that
\[ \e^{\mc L_{\kappa} t} =  T_{\kappa} (t) \quad \text { and } \quad \e^{\mc L_{\kappa}^*t} = S_{\kappa} (t) . \]
We are interested in convergence of these semigroups as $\kappa \to \infty$.

Note that in changing the diffusion matrix $A$ we are also changing the co-normal derivative
which appears in our transmission conditions. Thus, such a change results in speeding up the diffusion process while keeping the flux through the boundary constant. 
With the help of Theorem \ref{t.ouhabaz} we prove the following result. Note that the space $H_0$ appearing in the following 
theorem is closed in $L^2(\Omega_0)$, so that $H_0$ coincides with the closure of the form domain, considered in Section 
\ref{sec:prelim}.

\begin{thm}\label{t.conv}
As $\kappa\to\infty$
the form $\fa_\kappa$ converges in the strong resolvent sense to the restriction
of $\fq_0 + \fb$ to the domain
\[
H_0 = \{ u \in L^2(\Omega_0) : u|_{\Omega_k} \mbox{ is constant for } k=1, \ldots, N\}.
\]
Similarly, $\fa_\kappa^*$ converges in the strong resolvent sense to $\fq_0 + \fb^*$. Moreover, we have strong convergence
\[
T_{\kappa}(t)u \to \e^{-t (\fq_0 + \fb) }u \quad \mbox{and}\quad S_{\kappa}(t)u \to \e^{-t (\fq_0 + \fb^*)}u, \qquad t >0 ,
\]
in $L^2(\Omega_0)$ as $\kappa \to \infty$.
\end{thm}

\begin{proof}
We have $\Re \fa_\kappa[u] = \fq_\kappa [u] + \Re\fb[u]$ which is clearly increasing in $\kappa$.  On the other hand,
$\Im \fa_\kappa[u] = \Im \fb[u]$ since $\fq$ is symmetric. Thus $\Im\fa_\kappa[u]$ is independent of $\kappa$ whence
both contitions (b)(i) and (b)(ii) in Theorem \ref{t.ouhabaz} are satisfied. Next, since $\fa_1$ is sectorial we find a constant $C>0$ such that
\[
|\Im \fb [u]|\leq C(\fq[u] + \Re\fb[u]) \leq C(\kappa \fq [u] + \Re\fb[u]) = C\Re\fa_\kappa[u].
\]
This shows that the forms $\fa_\kappa, \kappa \ge 1$ are indeed uniformly sectorial so that all assumptions of Theorem \ref{t.ouhabaz} are satisfied, implying the strong resolvent convergence of the forms and strong convergence of the semigroups. 

Let us check that the limiting form  is as claimed. Obviously, $\sup_\kappa \Re\fa_\kappa[u] < \infty$ if and only if 
\begin{equation}\label{eq.alim}
\int_{\Omega_0} (A\nabla u)\cdot \overline{\nabla u}\, \dlam = 0
\end{equation}
 and in this case
  $\lim_{\kappa \to \infty} \fa_\kappa [u] = \fq_0[u]+\fb[u]$. Since the matrix $A$ is uniformly elliptic (as assumed throughout)
Equation \eqref{eq.alim} implies that $\nabla u =0$ on $\Omega_k$ for $k=1, \ldots, N$. As each $\Omega_k$ was assumed to be connected, $u$ is constant on each of these sets. 
Since, conversely, $u \in H_0$ implies \eqref{eq.alim}, we are done.
\end{proof}

\begin{rem}\label{l.markusa}
We can actually `speed up' diffusion in a much more general way and obtain the same convergence result. Indeed, let
$A_\kappa = (a_{ij}^{(\kappa)}) \in L^\infty(\Omega_0, \CR^{d\times d})$ be such that
\[
\gamma\| \xi\|^2 \leq \sum_{i,j=1}^d a_{ij}^{(\kappa)}(x) \xi_i\bar\xi_j \leq \sum_{i,j=1}^d a_{ij}^{(\kappa+1)}(x) \xi_i\bar\xi_j 
\]
for every $\kappa \ge 1$, every $\xi \in \CC^d$ and almost all $x \in \Omega_0$ and such that
\[
\sup_\kappa \sum_{i,j=1}^d a_{ij}^{(\kappa)}(x) \xi_i\bar\xi_j = \infty\]
 for almost all $x\in \Omega_0$ and all $\xi \in \CC^d\setminus \{0\}$. If we define the form $\fq_\kappa$ as in Definition \ref{def.forms} with $\kappa A$ replaced with $A_\kappa$, then
$\fq_\kappa$ is an increasing sequence of symmetric forms and $\sup \fq_\kappa [u] < \infty$ if and only if $\fq [u] = 0$, and the conclusion in Theorem \ref{t.l2conv} remains valid. 
\end{rem}

We next describe in more detail the limiting form and the limit semigroup, and provide a probabilistic interpretation of Theorem \ref{t.l2conv}. As we have seen in Section \ref{sec:prelim}, the limit semigroup basically operates on the space $H_0$, whereas
everything in $H_0^\perp$ is immediately mapped to 0.
The orthogonal projection onto $H_{0}$ is given by
\begin{equation}\label{eq.projection}
P_{H_{0}} u \coloneqq \sum_{k\in\mc N}\frac 1{\lam (\Omega_k)} \int_{\Omega_k} u\dlam \cdot \one_{\Omega_k}, \qquad u \in L^2(\Omega_0).
\end{equation}
Let $\mu$ be the measure on $\mc N$ (see \eqref{en}) defined by
\begin{equation}\label{miu}
\mu (S) = \sum_{k\in S}\lambda (\Omega_k), \quad \mbox{ for } S\subset \mc N.
\end{equation}
We denote the associated $L^2$ space by $\ell^2_\mu \coloneqq L^2(\mc N, 2^{\mc N}, \mu)$.  This space can be identified
with $\CC^N$ equipped with the norm
\[
\|x\|_{\ell^2_\mu} \coloneqq \Big( \sum_{k\in\mc N} |x_k|^2\lambda(\Omega_k) \Big)^{\frac{1}{2}}
\quad\mbox{for } x= (x_1, \ldots, x_N) \in \CC^N.
\]
Clearly, $\ell^2_\mu$ is a Hilbert space with respect to the scalar product
\[
\la x, y\ra_{\ell^2_\mu} = \sum_{k\in\mc N} x_k\overline{y_k}\lambda (\Omega_k).
\]
We note that the norm $\|\cdot\|_{\ell^2_\mu}$ is chosen in such a way that $\ell^2_\mu$ is isometrically isomorphic to $H_{0}$ viewed as a subspace of $L^2(\Omega_0)$, via the isomorphism
\[
\Phi: x\mapsto \sum_{k\in\mc N} x_k\one_{\Omega_k}.
\]
Under this identification, the number $\mu(\{k\})$ serves as a sort of weight for the $k$th component. When modeling the diffusion of some chemical substance, for example, in the limit equation the total mass of the diffusing substance in $\Omega_k$ is not the $k$th component $x_k$ of the vector $x$, but it is $ \mu (\{k\})\, x_k$. With this interpretation, the choice for the measure $\mu$ can be justified by observing that the set $\Omega_k$, which has measure $\lambda (\Omega_k) = \mu(\{k\})$ is in the limit lumped together into the single state $k\in \mc N$.

Our goal is to identify the operator associated with the limiting form, or -- more specifically --  its isomorphic image in $\ell^2_\mu.$  
To this end, for $ k \in \mc N, \ell \in \mc N_0$, $\ell\not = k$, let \[\rho _{k,\ell} = \int _{\Gamma_{k,\ell}}  \tau_k \ud \sigma. \] 
{For $\ell \neq 0$ this is} the total permeability of the membrane  $\Gamma_{k,\ell}$ separating $\Omega_k$ from $\Omega_\ell$ when approached from within $\Omega_k$. It may also be thought of as the average number of particles that filter through $\Gamma_{k,\ell}$ in a unit of time. Next, for $k \in \mc N$,
\[
\varrho_{k,k} = - \sum_{\ell\in \mc N_0, \ell \not = k} \rho_{k,\ell}\]
is (minus) the average number of particles that filter from $\Omega_k$ to an adjacent $\Omega_\ell$ in a unit of time, i.e., the number of particles lost by $\Omega_k$. Finally, the quantity \[ \varrho _{k,\ell} = \int _{\Gamma_{k,\ell}} b_{k,\ell} \tau_k \ud \sigma \qquad k,\ell \in \mc N,\] 
may be thought of as the average number of particles that after filtering from $\Omega_k$ to $\Omega_{\ell}$
 in a unit of time survive to continue their chaotic movement in $\Gamma_{\ell}$, i.e., the number of particles gained by $\Omega_\ell$ from $\Omega_k$. 
Finally, we define
\begin{equation}\label{qkl} q_{k,\ell} = \frac {\varrho_{k, \ell}}{\lam (\Omega_k)}, \qquad k, \ell \in \mc N \end{equation}
and let $Q $ be the real $n\times n$ matrix of the coefficients $q_{k,\ell}, k,\ell\in \mc N$.
Likewise, we define 
\begin{equation}\label{qklstar}
q_{k,\ell}^* = \frac{\varrho_{\ell,k}}{\lambda (\Omega_k)}, \qquad k, \ell \in \mc N
\end{equation}
and set $Q^*$ to be the real $n\times n$ matrix of the coefficients $q_{k, \ell}^*, k, \ell \in \mc N$. {Note that $Q^*$ is indeed the adjoint of the matrix $Q$ with respect to the scalar product $\langle \cdot, \cdot \rangle_{\ell^2_\mu}$, see  
\eqref{eq.q} and \eqref{eq.qstar} below,}  but it is different from the mere transpose $Q^\mathsf{T}$ of $Q$ which is the adjoint with respect to the canonical scalar product $\langle x, y \rangle \coloneqq
\sum_{k=1}^N x_k \bar y_k$.

In the case where $\tau_{k,0} =0$ for all $k \in \mc N$, i.e.,\ when we impose Neumann boundary conditions on the boundary of $\Omega_0$ and, additionally, $b_{k,\ell} = 1$ so that 
$\varrho_{k,\ell} = \rho_{k, \ell}$ for all $k,\ell \in \mc N$, i.e., when no loss of particles is possible in the process of filtering through the inward membranes,  
the diagonal entries in $Q$ are non positive, the off-diagonal entries are non negative and the row sums $\sum_{\ell \in \mc N} q_{k,\ell}$ are zero
for every $k\in \mc N$. This shows that $Q$ is the intensity matrix of a continuous time (honest) Markov chain with $N$ states. In general, however, a loss of probability mass is possible -- this corresponds to the possibility for a particle to be killed after filtering through the inward or outward membrane in the approximating process. Hence, in general, the chain described by $Q$ is not honest. 

\begin{prop} \label{opq} The operator associated with $\fb$ restricted to $H_0$ is 
\[ \Phi Q \Phi ^{-1},\]
the operator associated with $\fb^*$ restricted to $H_0$ is
\[ \Phi Q^* \Phi^{-1}.\] 
\end{prop}
\begin{proof} Let $u=\Phi (x), v=\Phi (y) \in H_0$.  Then \begin{align}
\fb[u,v] 
& = \sum_{\substack{k \in \mc N\\\ell\in \mc N_0}} \int_{\Gamma_{k,\ell}} \tau_k(x_k-b_{k,\ell}x_{\ell})\bar{y}_k \,\ud\sigma \notag\\
& = \sum_{\substack{k \in \mc N\\\ell\in \mc N_0}} x_k\overline{y_k} \int_{\Gamma_{k,\ell}} \tau_k\, \ud\sigma - \sum_{\substack{k \in \mc N\\\ell\in \mc N_0}}
x_{\ell}\overline{y_k}  \int_{\Gamma_{k,\ell}} \tau_k b_{k,\ell} \ud\sigma  \notag\\
& = -\sum_{k\in\mc N} \varrho_{k,k}x_k\overline{y_k} - \sum_{\substack{k,\ell\in \mc N\\k\neq \ell}}
\varrho_{k,\ell} x_{\ell} \overline{y_k} \qquad \text{(recall $b_{k,0}=0$ and $\Gamma_{k,k}=\emptyset $)}\notag\\
& = -\sum_{k\in \mc N} \sum_{\ell\in\mc N}\frac{\varrho_{k,\ell}}{\lambda (\Omega_k)}x_{\ell} \overline{y_k}\lambda(\Omega_k) =- \langle Qx,y \rangle_{\ell^2_\mu} , \label{eq.q}\end{align} 
where $Qx$ is the matrix product, or (changing the order of summation)
\begin{align}
& = - \sum_{\ell\in\mc N} \sum_{k\in \mc N} x_\ell \overline{\frac{\varrho_{k,\ell}}{\lambda (\Omega_\ell)}y_k}\lambda(\Omega_\ell) =- \langle x,Q^*y \rangle_{\ell^2_\mu},  \phantom{============} \label{eq.qstar}
\end{align}
where $Q^*y$ is the matrix product. It follows that 
\[ \fb[u,v] = - \langle Q\Phi^{-1}u, \Phi^{-1}v \rangle_{\ell^2_\mu} = - \langle \Phi Q\Phi^{-1}u, v \rangle_{H_0},  \] 
where the last scalar product in $H_0$ is that inherited from $L^2(\Omega_0).$ {Likewise
\begin{align*}
\fb^*[u,v] & = \overline{\fb [v,u]} =  -\langle Q^*\Phi^{-1}u, \Phi^{-1}v\rangle_{\ell^2_\mu}
= -\langle \Phi Q^*\Phi^{-1} u, v\rangle_{H_0}.
\end{align*}}
This completes the proof. \end{proof}

Let us now take care of the potential term.
We put $C \coloneqq \diag (\Phi^{-1}Pc)$, where $P$ is defined by \eqref{eq.projection}. In other words, $C$ is the diagonal matrix whose entries are the average values 
of $c$ on the sets $\Omega_k$ ($1\leq k\leq N$). A straightforward computation shows that 
the operator related to $\fq_0 $ restricted to $H_0$ is $- \Phi C \Phi^{-1}.$ As $C$ is a diagonal matrix, we see that $C^* = C$,
whence the operator related to $\fq_0^*$ restricted to $H_0$ is also $-\Phi C \Phi^{-1}$.
Combining this with Proposition \ref{opq}, we obtain the following corollary. 

\begin{cor} The operator associated with the limiting form $\fq_0+\fb$ (resp.\ $\fq_0^* + \fb^*$) on the domain $H_0$ is 
\[ \Phi (Q - C) \Phi^{-1} \quad (\mbox{resp. } \Phi (Q^*-C)\Phi).\] \end{cor}

From now on, we no longer distinguish between $\ell^2_\mu$ and its isometric image $H_0 = \Phi(\ell^2_\mu)$. Thus, with slight abuse of
notation, we will consider the (matrix) semigroups $\e^{t(Q-C)}$ and $\e^{t(Q-C)^*}$ as semigroups on the space $H_0$.
With this convention, we can now reformulate Theorem \ref{t.conv} as follows (see \eqref{eza} and \eqref{eq.projection}). 

\begin{thm}\label{t.l2conv} 
With the notation introduced above, we have for every $u\in L^2(\Omega_0)$ and $t>0$
\begin{align*}
\lim_{\kappa\to\infty} T_{\kappa}(t) u & = \e^{t (Q-C)} P_{H_{0}} u  \quad\mbox{and}\\
\lim_{\kappa\to\infty} S_{\kappa} (t) u & = \e^{t (Q^* - C)} P_{H_{0}} u.
\end{align*} \end{thm}

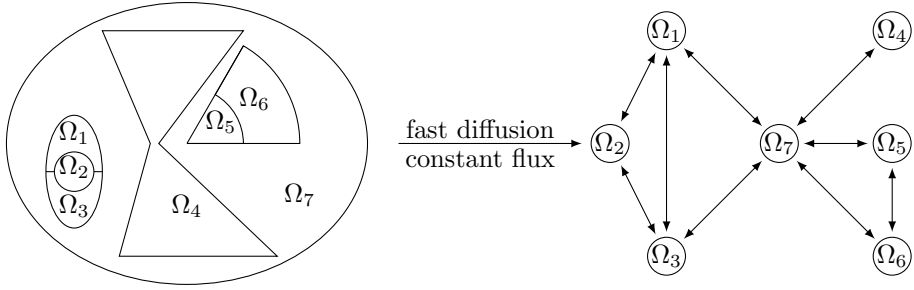
\begin{figure}
\begin{tikzpicture}[scale=0.75]
   \draw (0:0cm) -- (0:2cm)
         arc (0:60:2cm) -- (60:1cm)
         arc (60:0:1cm) -- (60:0cm)
-- cycle;\draw (0,0) -- (1,1.732);
\draw (0,0) ellipse (3.2cm and 2.5cm);
\draw (-2,-0.5) ellipse (0.5cm and 1 cm)
               ellipse (0.35cm and 0.35cm);
\draw (-0.65,0) -- (-1.2,-2) -- (1.6,-2) -- (-0.5,0) -- (1,2)-- (-1.5,2)--cycle;
\draw (-1.65,-0.5)--(-1.5,-0.5);
\draw (-2.5,-0.5)--(-2.35,-0.5);
\node [above] at (-2,-0.1) {$\Omega_1$};
\node [above] at (-2,-0.7) {$\Omega_2$};
\node [above] at (-2,-1.4) {$\Omega_3$};
\node [above] at (-0,-1.4) {$\Omega_4$};
\node [above] at (2,-1.2) {$\Omega_7$};
\node [above] at (0.6,0.1) {$\Omega_5$};
\node [above] at (1.2,0.5) {$\Omega_6$};
\draw [->](3.75,0) -- (7.0,0); 
\node [above] at (5.2,0) {\text{fast diffusion}};
\node [below] at (5.2,0) {\text{constant flux}};
  \tikzstyle{every node}=[draw,shape=circle];
 \path (10.5,0) node (v0) {$\Omega_7$};
\path (8.5,2) node (v1) {$\Omega_1$};
\path (7.5,0) node (v2) {$\Omega_2$};
\path (8.5,-2) node (v3) {$\Omega_3$};
   \path (12.5,2) node (v4) {$\Omega_4$};
   \path (12.5,0) node (v5) {$\Omega_5$};
\path (12.5,-2) node (v6) {$\Omega_6$};
  \draw [<->](v0) -- (v1);
  \draw      [<->] (v1) -- (v3);
  \draw      [<->] (v0) -- (v3);
  \draw       [<->] (v0) -- (v4);
  \draw       [<->](v0) -- (v5);
 \draw [<->] (v0) -- (v6);
 \draw [<->] (v1) -- (v2);
 \draw [<->] (v3) -- (v2);
 \draw [<->] (v5) -- (v6);

 \end{tikzpicture} \caption{State-space collapse as fast diffusions on domains separated by semi-permeable domains converge to a Markov chain. Intensity of jump between aggregated states $\Omega_k$ and $\Omega_\ell$ is  $q_{k,\ell} = \frac {\varrho_{k, \ell}}{\lam (\Omega_k)}$.}\label{fig3}\end{figure}

To summarize: in probabilistic terms discussed in this subsection, Theorem \ref{t.l2conv} with $c=0$ asserts that our diffusion processes converge to a continuous time Markov chain with state space $\mc N$ which may be though of as being composed of $N$ aggregated states, each of them corresponding to one domain of diffusion (see Figure \ref{fig3}); $c\not =0$ plays the role of a potential term. As advertised in the introduction, the jump intensities in this chain (given by \eqref{qkl}) are  in direct proportion to the total permeability of the membranes, and in inverse proportion to the sizes of the domains.

In our examples in the following section, we will also consider \emph{inhomogeneous equations}, i.e.,\ equations of the form
\[
z'(t) = Az(t) + f(t),\quad u(0) = u_0
\]
where $A$ is the generator of a strongly continuous semigroup $T$ on a Banach space $X$ and $f \in L^1( (0,t_0);X)$, for some $t_0>0.$ 
Most often, one uses the concept of a \emph{mild solution}
for such equations. By \cite[Proposition 3.1.6]{abhn}, the mild solution is given through the \emph{variation of constants formula}
\begin{equation}\label{eq.voc}
z(t) = T(t)u_0 +\int_0^t T(t-s)f(s)\, \ud s
\end{equation}
for $t \in [0,t_0]$.

\begin{cor} \label{cor.main} Fix $t_0 >0.$ Let $[0,t_0] \ni t \mapsto f(t)\in L^2(\Omega_0)$ be a Bochner integrable function, and $u_0$ be a fixed element of  $L^2 (\Omega_0)$. Then, as $\kappa \to \infty$, the mild solutions $[0,t_0] \ni t \mapsto z_\kappa (t) \in L^2(\Omega_0)$ of the Cauchy problems 
\[ z_\kappa'(t) = \mc L_{\kappa}u_\kappa (t) + f(t),  \quad t \in [0,t_0], \quad z_\kappa (0) = u_0, \] converge 
in $L^2(\Omega_0)$ and pointwise on $(0,t_0]$ to the function
\[ e^{t(Q-C)}P_{H_0}u_0 + \int_0^t e^{(t-s)(Q-C)}P_{H_0}f(s)\ud s, \] 
the solution of the Cauchy problem
\[
z'(t) = (Q-C) z(t) + P_{H_0}f(t)\quad t \in [0,t_0], \quad z (0) = P_{H_0} u_0
\]
on the space $H_0$. An analogous result holds for $\mc L_{\kappa}^*.$
\end{cor}

\begin{proof} 
This is immediate from Theorem \ref{t.l2conv}, formula \eqref{eq.voc} and the dominated convergence theorem.
\end{proof}

\section{Examples}\label{sec:6}

\subsection{Kinase activity} \label{kinase}
Fast diffusion is a rich source of interesting singular perturbations, see e.g. \cite{osobliwiec}. This is the case, for example, in the following model of kinase activity from \cite{kazlip} (see also \cite{kazlip2}). Let us recall that kinases are enzymes that transport phosphate groups. In doing this, protein kinases transmit signals and control complex processes in cells. In \cite{kazlip}, following \cite{brown}, a cell is modeled as a unit $3d$ ball. All kinases, whether active (i.e., phosphorylated) or inactive, are diffusing inside the ball. Binding a receptor located at the cell membrane (the sphere) by an extracellular ligand is a signal which is to be conveyed to the cell. This is done by the kinases which, when touching the boundary (the sphere) become activated by their interaction with the 
ligand-bound receptors; 
such active kinases diffuse freely into the interior 
of the cell. Simultaneously, they are randomly inactivated when meeting phosphatases which are uniformly distributed over the cell.

\newcommand{\kas}{K^\star}
\newcommand{\kac}{K^{\diamond}}
\newcommand{\mkas}{k^\star}
\newcommand{\mkac}{k^{\diamond}}

In the no feedback case, where all receptors at the membrane are ligand-bound almost simultaneously, reaching a uniform stable concentration $C>0$, the master equation for the concentration $\kas $ of active kinases (after suitable rescaling) is a diffusion-degradation equation 
\begin{equation}\label{ki1a} \frac {\partial \kas }{\partial t} = \kappa \triangle \kas - \kas , \qquad t \ge 0, \end{equation}
with boundary condition 
\begin{equation}\label{ki1b} a C (1 - \kas_{|0} ) = \kappa \frac {\partial \kas}{\partial \nu  }.\end{equation}
Here, $\kappa>0$ is a diffusion coefficient and $a>0$ is a reaction coefficient. $\kas_{|0}$ is the value of $\kas$ at the boundary, $\frac {\partial \kas}{\partial \nu  }$ is the usual normal derivative at the boundary 
and the term $-\kas $ describes random dephosporylation of active kinases. We note that condition \eqref{ki1b} describes an inflow of active kinases from the boundary 
(this boundary condition is missing in \cite{brown} and was introduced in \cite{kazlip}).

One of the aims of both \cite{brown} and \cite{kazlip} is to show that (perhaps somewhat surprisingly) \emph{slow} diffusion may facilitate signal transmission more effectively than fast diffusion. To show this, the authors of \cite{kazlip} 
study the case of infinitely fast diffusion and compare the properties of solutions of the limit equation with those of the original one, 
showing that the infinite diffusion case leads to less effective signal transmission.
To do this, they assume spherical symmetry and argue that the limit equation has to be of the form 
\begin{equation} \label{ki3} 
\frac {\ud \kas}{\ud t} = 3a C (1- \kas ) - \kas , \qquad t\ge 0.\end{equation}
This is interpreted as follows: As diffusion coefficients increase to infinity, the active kinases' distribution becomes uniform over the ball and may be identified with a real function of time, whose dynamics is then described by \eqref{ki3}. 
Nevertheless, the form of the limit equation is quite curious, with particularly  intriguing factor $3$.

In \cite{osobliwiec} a convergence theorem for semigroups on the space of continuous functions has been proved asserting that,
if spherical symmetry suggested in \cite{kazlip} is granted, the solutions of  \eqref{ki1a}--\eqref{ki1b} indeed converge to those to \eqref{ki3} equipped with appropriate initial condition. Here we will show that this convergence for $\kappa \to \infty$ is a special case of our averaging principle. In fact, we can also prove convergence in more general situations, where the ball is replaced by an arbitrary bounded domain, the Laplacian is replaced by a more general diffusion operator and we can consider more general Robin boundary conditions. Note, however, that in our context we obtain convergence in the sense of $L^2$, and not in a space of continuous functions, thus we do not obtain uniform convergence.
   
To fit the above example into our framework, we only need the simplest situation where $\Omega_0$ is not partitioned into 
subregions. Thus, we have $N=1$ and $\Omega_N = \Omega_0$.
Strictly speaking, therefore, in this case we are not dealing with transmission conditions, but merely with (Robin) boundary conditions. On the other hand, boundary conditions may be seen as particular instances of transmission conditions. In other words, the outer boundary of $\Omega_0,$ i.e., the boundary of $\Omega_0$ with its complement, may be thought of as a membrane that is permeable only in one direction.
  
Thus, let $\Omega_0$ be a bounded domain in $\CR^3$ with Lipschitz boundary $\Gamma_0$. We replace
the operator $\kappa \triangle - I$ with $\mc L_{\kappa}$, the $L^2(\Omega_0)$ version of the elliptic operator \eqref{operatorl} with diffusion matrix $A$ replaced by $\kappa A.$  
Moreover, instead of the constant $aC$ in the boundary condition, we consider a non-negative function $\tau \in  L^\infty(\Gamma_0; \CR)$ (playing the role of $\tau_0$ of Section \ref{dpm}). With these generalizations, equations \eqref{ki1a}--\eqref{ki1b}  become
\begin{equation}\label{ki2a} \frac {\partial \kas }{\partial t} = \mc L_{\kappa} \kas , \qquad t \ge 0, \end{equation}
and 
\begin{equation}\label{ki2b} \tau (1 - \kas_{|0} ) = N_0(\kas).\end{equation}
Note that $N_0$ is now the conormal derivative with respect to the matrix $\kappa A$, and so the constant $\kappa$ is no longer visible on the right-hand side of the boundary condition. As in \cite{kazlip}, we are interested in the limit as $\kappa \to \infty$.

To transform this system to a form suitable for application of Theorem \ref{t.l2conv} we consider $\kac$, the concentration of inactive kinases, defined as
\[ \kac = \one_{\Omega_0} - \kas .\]   A straightforward calculation shows that $\kac$ satisfies:
\begin{equation*} \frac {\partial \kac }{\partial t} = \mc L_{\kappa} \kac + \one_{\Omega_0}, \qquad t \ge 0, \end{equation*}
and 
\begin{equation*}- \tau \kac_{|0} = N_0(\kac),\end{equation*}
i.e., an equation of the type Corollary \ref{cor.main} is devoted to (with $\kac$ playing the role of $z_\kappa$).

Since in this case $N=1$, the orthogonal projection onto $H_{0}$ is just the orthogonal projection onto the constant functions given
by $Pu = \frac{1}{\lam (\Omega_0)} \int_{\Omega_0} u \dlam$. Also, the matrices $Q$ and $C$ are real numbers given by
\begin{equation}\label{qu} -q \coloneqq q_{1,1} = -\frac 1{\lam (\Omega_0) } \int _{\Gamma_{0}} \tau_0 \ud \sigma\quad \text{ and } \qquad c_{1,1} = 1,\end{equation}
respectively. Thus, $Q-C = -(q+1)$ and, as a consequence of Corollary \ref{cor.main}, in the limit as $\kappa \to \infty$, 
$\kac (t)$ converges strongly to $\mkac (t) \one_{\Omega_0}$, where $\mkac $ is the  solution of 
\begin{equation*} \frac {\ud \mkac}{\ud t}  = -(q+1) \mkac  + 1, \qquad t \ge 0, \end{equation*}
with initial condition $\mkac (0) = P(\one_{\Omega_0} -K_0^*)= 1 - P\kas_0 $, where $\kas_0$ is the initial concentration of $\kas $. Therefore, $\kas (t)$ converges to $\mkas  (t)\one_{\Omega_0} $ where $\mkas $ is the solution of 
\[ \frac{\ud \mkas }{\ud t} = q (1- \mkas ) - \mkas  , \quad t \ge 0, \qquad \mkas(0)= P\kas_0 .\]
This indeed generalizes the results from \cite{kazlip} and \cite{osobliwiec}, because in the case where $\Omega_0$ is the $3d$ unit ball and $\tau_0$ is the constant function equal to $aC$, we have by \eqref{qu}:
\[ q = aC \, \frac {\text{unit ball's surface area}}{\text{unit ball's volume}} = 3aC. \] 

\subsection{Neurotransmitters} \label{neuro} In modeling dynamics of synaptic depression, one often adopts a widely accepted, if simplified, view that a secretory cell is divided into three subregions, $\Omega_1, \Omega_2$ and $\Omega_3$ corresponding to the so-called immediately available, small and large pools, where neurotransmitters are located. This is also the case in the model of Bielecki and Kalita \cite{bielecki}, in which a terminal bouton, playing the role of our $\Omega_0$, is modeled as a $3d$ region (see Figure \ref{figure2}) and the concentration of (vesicles with) neurotransmitters is described by functions on those subregions. However, no clear distinction between the subregions is made; in particular, no transmission conditions on the borders between pools are imposed, and it appears as if diffusing vesicles with neurotransmitter may freely cross from one pool to the other. For reasons explained in \cite{bobmor}, such a model cannot be easily connected with the older, and apparently better known model of Aristizabal and Glavinovi\v{c} \cite{ag} where the situation is described by three scalars, evolving with time, i.e., by the levels $U_i, i =1,2,3$ 
of neurotransmitters in the pools $\Omega_1$, $\Omega_2$ and $\Omega_3$ respectively. 
Arguably, to draw such a connection, specifying the way the particles may filter from one region to the other is necessary, and it transpires that the appropriate transmission conditions are of the form \eqref{eq.trans} with $b_{k,\ell}\equiv 1$. (One should note here, however, that no physical membranes separating pools exist in the secretory cells, and the interpretation similar to the Newton's Law of Cooling seems to be more suitable, for example the one provided by Fick's law.)

To see that the connection in question is a particular case of our averaging principle, we rewrite the governing equation for the neurotransmitter level $u$  in $L^2(\Omega_0)$  in the form (compare \cite[eq.\ (1)]{bielecki}): 
\begin{equation}\label{neu1} \frac {\partial u }{\partial t} = \mc L_{\kappa} u + \beta u^\sharp, \qquad t \ge 0, 
\end{equation}
where $\beta$ is a measurable, bounded and non-negative function which vanishes everywhere but on $\Omega_3$, and is interpreted as neurotransmitter's production rate (varying in $\Omega_3$), $\mc L_{\kappa}$ is the $L^2(\Omega_0)$ version of the elliptic operator \eqref{operatorl} with $c=\beta$ and diffusion matrix $\kappa A$, and $u^\sharp \in L^2(\Omega_0)$ is a given function interpreted as a balance concentration of vesicles.

Clearly this governing equation is of the form considered in Corollary \ref{cor.main} with $f(t) = \beta u^\sharp $ independent of $t.$ In this case, the space $H_0$ is composed of functions that are constant on each of the three pools (separately), and the projection on this space is a particular case of \eqref{eq.projection} with $\mc N=\{1,2,3\}$. As explained in Section \ref{sect.conv}, $H_0$ is isometrically isomorphic to $\CC^3$ with suitable norm. In particular, since $\beta$ vanishes on $\Omega_1$ and $\Omega_2$, the function
$P \beta u^\sharp $ may be identified with the vector $e = (0,0,e_3) \in \CC^3$ where 
\[  e_3 = \frac 1 {\lam (\Omega_3)} \int_{\Omega_3} \beta u^\sharp \ud \lam . \]
Hence, identifying isomorphic objects, we see that Corollary \ref{cor.main} establishes convergence of solutions of \eqref{neu1} to a $\CC^3$-valued function $u$ solving the equation: 
\[ u'(t) =  (Q- C)u(t) + e .\]

In order to find a more explicit form of the limit matrix $Q$ we note that, because of the special arrangement of pools, $\Omega_3 $ borders only with $\Omega_2$, and $\Omega_1$ is the only region having common border with the complement of $\Omega_0$. As a result (see \eqref{qkl} and recall that we agreed on $b_{k,\ell}\equiv 1$)
 \[ Q = \begin{pmatrix} -q_{10} - q_{12} & q_{12} & 0 \\
q_{21} & - q_{21} - q_{23} & q_{23} \\
0 & q_{32} &- q_{32} \end{pmatrix}, \qquad q_{k,\ell} = \frac {\int_{\Gamma_{k,\ell}}\tau_{k}\ud \sigma }{\lam(\Omega_k)}. \]
Also, diagonal entries of the matrix $C$ are the values of $P\beta$ on the sets $\Omega_1, \Omega_2$
 and $\Omega_3$, respectively, and  since $\beta$ vanishes on $\Omega_1$ and $\Omega_2$, it follows that $C$ acts on a vector
 in $\CC^3$ as coordinate-wise multiplication with the vector $(0,0,c)^{\mathsf T}$, where
 $c= \frac{1}{\lam (\Omega_3)}\int_{\Omega_3} \beta\dlam$.  

So, the limit equation is precisely of the form considered by  Aristizabal and Glavinovi\v{c} for the levels $U_i, i =1,2,3$ (which now may be thought of as coordinates of $u$).  Comparing the entries of the matrix $Q - C$ with the coefficients used by Aristizabal and Glavinovi\v{c}, we may interpret the latter in new terms, see the discussion given in \cite{bobmor} and  compare with eq.\ (7) there. (The apparent discrepancy between our $Q$ and that given in the cited equation (7) is that the latter involves diffusion coefficients. To explain this, we note that transmission conditions in \cite{bobmor} are devised in a slightly different way than here. In particular, in \cite{bobmor} it is not the flux but the ratio flux/diffusion coefficient that is preserved.)

\subsection{Intracellular calcium dynamics}\label{icd} The last example concerns calcium dynamics in eukaryotic cells. 
Calcium plays a crucial role in mediating and recognising signals from the extracellular space 
into various parts of the cell, in particular to the nucleus. 
On the other hand, an elevated concentration  
of calcium ions inside the cytosol is harmful and may induce the cell's apoptosis. 
For that reason it is stored also in intracellular compartments, like 
endoplasmic reticulum or mitochondria. 
(A large amount of calcium is also bound to so-called buffer protein molecules.) 
The  average concentration of free calcium 
inside the cytosol does not exceed 1 \,$\mu M$, while the average concentration of 
calcium inside endoplasmic reticulum and mitochondria may be two orders of magnitude 
bigger \cite{keener}. This is possible due to the action of special pumps, which by using 
different forms of energy can push free calcium into the regions of 
higher concentration, e.g. SERCA pumps (reticulum) or 
mitochondrial sodium-calcium exchangers (MNCX). In this way, cells 
can transport calcium against the diffusional flux.

In some circumstances, oscillations of calcium concentration between the internal stores 
and cytosol are observed. Such oscillations are usually described by means 
of systems of ordinary differential equations (see, e.g., \cite{keener,Marhl,Dyzma}). In these descriptions, inhomogenities in the spatial  distribution of calcium inside the regions 
corresponding to different cell compartments are neglected. 
As we will argue,  our Theorem \ref{t.l2conv} justifies such a simplified description, provided diffusion in the cell is fast.   

To begin with, we assume that the processes of binding and unbinding of calcium ions by buffer 
molecules, characterized by certain parameters 
$k_+>0$ and $k_->0$, respectively, are very fast. This allows applying
the reduction method of Wagner and Keizer \cite{Tsai_Sneyd}, so that equations for 
buffer molecules are neglected. Moreover, for further simplicity, we assume that the 
buffers are immobile, that is to say their diffusion coefficients 
are negligible and that they are uniformly distributed in the space.

Let $\Omega_0 \subset \CR^3$ model the spatial region occupied by the cell, with the exception of its nucleus.  
Let  $\Omega_{1} \subset  \Omega_0$ correspond to the 
region occupied by the endoplasmic reticulum of the cell,  and let  
$\Omega_2 \subset \Omega_0$,  disjoint from $\Omega_1$, correspond to the mitochondria inside the cell. (Usually, there is a number of mitochondria, but to simplify the model we combine the regions occupied by them into a single region.) Finally, let 
\(
\Omega_3 \coloneqq \interior \big( \Omega_0 \setminus \bigcup_{k=1}^{2}\overline{\Omega}_k \big),
\)
correspond to the cytosolic region of the cell (comp. Figure \ref{calcium}).

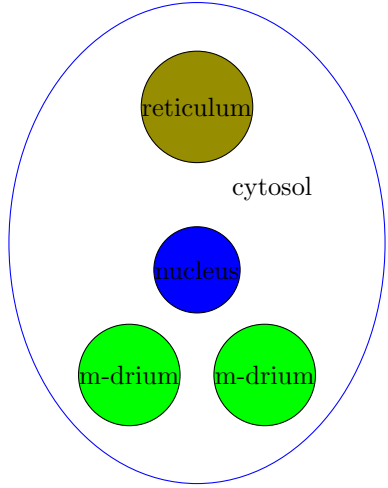
\begin{figure}\begin{tikzpicture} \draw [color=blue] (0,0) ellipse (2.5cm and 3.2cm);
\node [above] at (1,0.5) {\text{cytosol}};
\tikzstyle{every node}=[draw,shape=circle,fill=blue];
\node [above] at (0,-1) {\text{nucleus}};
\tikzstyle{every node}=[draw,shape=circle,fill=green];
\node [above] at (0.9,-2.5) {\text{m-drium}};
\node [above] at (-0.9,-2.5) {\text{m-drium}};
\tikzstyle{every node}=[draw,shape=circle,fill=olive];
\node [above] at (0,1) {\text{reticulum}};

\end{tikzpicture} \caption{The spatial region occupied by the cell with the exception of nucleus (where calcium cannot be stored) corresponds to $\Omega_0$; endoplasmic reticulum corresponds to $\Omega_1$, mitochondria correspond to $\Omega_2$, and $\Omega_3$ is the cytosol. $\Gamma_0=\Gamma_{0,3}$ is equal to the union of  the ellipse and the smallest circle in the center. $\Gamma_2 = \Gamma_{2,3} $ is the union of two circles in the bottom, and $\Gamma_1 = \Gamma_{1,3}$ is the circle at the top.}   \label{calcium} \end{figure}


The concentration  $U$ of free calcium in 
$\Omega_0$ is governed by the following equation \cite{Tsai_Sneyd}: 
\begin{equation} \label{initcalcium}
\frac{\partial U}{\partial t} = a(x) \frac{1}{1+\alpha(x,U)} \triangle U
\end{equation}
where the  positive function $a \in L^\infty ( \Omega_0) $ describes the diffusivity 
of the free calcium ions, i.e., of the ions which are not bound 
to buffer molecules. Since diffusivity does not change within each region, we assume that $a = \sum_{i=1}^3 a_i \one_{\Omega_i}$ for some positive $a_i, i=1,2,3$.

In \eqref{initcalcium}, the factor $\eta(x,U) = \frac{1}{1+\alpha(x,U)}$ where $\alpha $ is a non-negative function, describing the effect of calcium buffering for sufficiently large 
binding and unbinding coefficients $k_+$ and $k_-$, comes into play 
as a result of the Wagner and Keizer reduction method. 
{Let us also note} that due to the fact that the buffer molecules are assumed 
to be immobile, the gradient quadratic term in equation (2.5a) in  \cite {Tsai_Sneyd} 
vanishes. For simplicity we confine ourselves to the case of one representative 
kind of buffers in each of the regions $\Omega_1$, $\Omega_2$ and $\Omega_3$. We allow the coefficients
$k_+$ and $k_-$ to differ in the subregions of the cell, i.e.,\, for them to be functions of $x$. We set $K = k_-/k_+$.
Then, $\alpha $ is given by \[\alpha (x) = b_{\rm{tot}}(x) {K}(x)({K}(x) + U)^{-1},\] where 
$b_{\rm{tot}}$ denotes the total concentration of 
buffering molecules. In general, {also} $b_{\rm{tot}}$ depends on $x$ .
In some situations, however, it may be assumed that 
$\alpha$ does not depend on $U$, and depends on $x$ only via $\Omega_i$. Such an approximation can be justified in the cytosolic 
region by the fact that for typical endogeneous buffers we have 
$k_+ \approx 50 \mu M^{-1} s^{-1} $, $k_- \approx 500 s^{-1} $, so that ${K} \approx 10 \mu M$, 
and the maximal value of $U$ of the order of $1 \mu M$. 
On the other hand, the calcium capacity of reticular and mitochondrial subregions
is very large, so the concentration of calcium $U$ does not change 
significantly in these compartments in the non-apoptotic state of the cell. 
To be able apply the theory developed in this paper, we further simplify the model and assume that {$\eta < 1$} is a constant.

We are  thus lead to the following, reduced form of equation \eqref{initcalcium}:
\begin{equation} \label{initcalcium2}
\frac{\partial U}{\partial t} = \eta a \triangle U;
\end{equation}
in particular, the diffusion matrix $A(x)$ is a $3\times 3$ diagonal matrix with all entries on the diagonal equal $\eta a_i $ in 
$\Omega_i$.

Turning to transmission conditions, we assume -- in accordance with biological reality -- that neither the reticulum nor the mitochondria have common points with the cell's  membrane $\Gamma_0$, and that they do not communicate with each other directly, either. As a result, calcium may only permeate from the cytosol to  reticulum or mitochondria and back to cytosol, or from the cytosol to the extracellular matrix.   
Secondly, we suppose that the functions 
describing the flows through the separating membranes are {\it linear}. (We note, however, that the process of calcium transmission through the cell membrane, as well as 
that through the reticular and mitochondrial boundaries is rather complicated and these functions are, in general, nonlinear. A possible form of the functions 
modulo constant factors can be deduced from \cite{Marhl}, where a three-compartmental 
non-spatial model of calcium dynamics is proposed. An extension of the main theorem of our paper to the case of nonlinear transmission conditions is 
a very interesting topic for future research.) We thus suppose that the transport of calcium through the membranes separating the reticular and mitochondrial subregions from the cytosol  
is governed by the transmission conditions: 
\begin{align} 
\tau_{3} u_{|3} - \tau_1 u_{|1} &= a_1 \frac{\partial u}{\partial \nu},  \qquad  \tau_{3} u_{|3} - \tau_1 u_{|1} = a_3  \frac{\partial u}{\partial \nu}, \quad \text{ on } \Gamma_1 = \Gamma_{1,3} \quad \text{ and }  \nonumber \\
\tau_{3} u_{|3} - \tau_2 u_{|2} &= a_2\frac{\partial u}{\partial \nu},   \qquad  \tau_{3} u_{|3} - \tau_2 u_{|2}= a_3 \frac{\partial u}{\partial \nu}, \quad \text{ on } \Gamma_2 = \Gamma_{2,3} \label{rhomu}
\end{align}
respectively, where $\tau$'s are permeability functions, as in \eqref{newtc}. {Recall that 
$\tau$'s in general depend on $x$.} 

Several remarks are here in order. First of all, we note that since \eqref{initcalcium2} is to describe distribution of calcium ions, we use transmission conditions akin to \eqref{newtc} and not \eqref{eq.trans} (see Remark \ref{dualnosc}). Secondly, the first of the equations in the first line describes the flux
of calcium from the cytosol to the reticulum, while the second describes
the flux of calcium from the reticulum to the cytosol. Thirdly, {in the
first of these equations, $\frac{\partial u}{\partial \nu}$ refers to the
derivative of $u$ in direction of the outer normal of $\Omega_1$, whereas
in the second it refers to the derivative of $u$ in direction of the inner
normal of $\Omega_1$. Note that the inner normal of $\Omega_1$ is the
outer normal of $\Omega_3$.}
The other two equations are interpreted in the same way. 

Additionally, we have the equation
\begin{equation}\label{ccc} -\tau_{3} u_{|3} 
= a_3 \frac{\partial u}{\partial \nu} \quad {\rm on} ~  \Gamma_{0} = \Gamma_{0,3},\end{equation}
governing the 
outflow of free calcium ions through the outer boundary of the cell and through the membrane separating the cytosol from the cell's nucleus. Since the latter part of $\Omega_{0}$ is impermeable for the ions, we assume that $\tau_3 $ vanishes there. 
We note that condition \eqref{ccc} implies that we assume {either that the local free calcium concentration in the extracellular space is zero or the influx of calcium from outside the cell is blocked.} 



We stress that transmission conditions \eqref{rhomu} are not yet of the form \eqref{newtc}, because the right-hand sides are not yet the conormal derivatives for the operator $\eta a \triangle $ appearing in \eqref{initcalcium2}. Biologically, this is a reflection of the fact that only free calcium ions can 
pass through the separating boundaries --   
the buffer molecules (either free or with bound calcium ions) cannot do that. 
Mathematically, to make \eqref{rhomu} 
 compatible with \eqref{eq.trans} we need to multiply all equations by $\eta<1$. 
As we are discussing a model for the densities, it is the adjoint $Q^*$  defined by
\[  Q^* = \eta \begin{pmatrix}  - \frac{ \int_{\Gamma_1} \tau_1 \ud \sigma }{\lam (\Omega_1)} & 0 &
\frac{ \int_{\Gamma_1} \tau_3 \ud \sigma }{\lam (\Omega_1)} \\
0 & - \frac{ \int_{\Gamma_2} \tau_2 \ud \sigma }{\lam (\Omega_2)} & \frac{\int_{\Gamma_2} \tau_3 \ud \sigma }{\lam (\Omega_2)} \\
 \frac{ \int_{\Gamma_1} \tau_1 \ud \sigma }{\lam (\Omega_3)} &  \frac{\int_{\Gamma_2} \tau_2 \ud \sigma }{\lam (\Omega_3)} &  - \frac{  \sum_{i=0}^2 \int_{\Gamma_i}  \tau_3 \ud \sigma} {\lam (\Omega_3)} \end{pmatrix}, \] 
that governs the evolution in the limit. 
Hence, in this approximation, buffer molecules' influence reduces to slowing down the process of communication between reticulum, mitochondria and cytosol. Remarkably, more interesting phenomena are observed even for $\eta$ dependent on $x$ merely via $\Omega_i.$ 

Theorem \ref{t.l2conv} asserts that if $a$ is replaced by $\kappa a$ or if (see Remark \ref{l.markusa}) $a^{(\kappa)}$ is a family of functions indexed by $\kappa$ such that $\sup_{\kappa } a^{(\kappa)} (x) = \infty$  for almost all $x\in \Omega_0$,  then as $\kappa \to \infty$ solutions to \eqref{ugran} become more and more uniform, i.e., `flat',  in each of the regions $\Omega_1,\Omega_2,\Omega_3$. Moreover, if $u_i(t)$ denotes the common value of the limit function at time $t$  in $\Omega_i$ then for the column vector $u(t) $ with coordinates $u_i(t)$ we have
\begin{equation} \label{ugran} u'(t) = Q^* u(t) . \end{equation} 
Alternatively, $v (t) $ with coordinates $v_i(t) = \lambda (\Omega_i) u_i(t) $ (total probability mass in the $i$th region) satisfies 
\begin{equation*}  v'(t) = Q^\mathsf{T} v(t) , \end{equation*}   
where $Q^\mathsf{T}$ is the transpose of $Q$.

The form of the limit system (\ref{ugran}) agrees with the following heuristic reasoning. Suppose that $u (t) =\sum_{i=1}^3 u_i(t) \one_{\Omega_i}$ is a solution to 
\eqref{initcalcium2} with transmission conditions \eqref{rhomu} and \eqref{ccc} and $a$ replaced by $\kappa a.$  Then, using the Gauss theorem, we see that 
\[ \lam (\Omega_1) u_1'(t) = - \int_{\Gamma_1} \eta \tau_{1} \ud \sigma \, u_1(t) + \int_{\Gamma_1} \eta \tau_{3} \ud \sigma \, u_3(t), \] 
and dividing by $\lam (\Omega_1)$ leads to the first equation in \eqref{ugran}.
Similarly, we check that the second and third equations of (\ref{ugran}) 
agree with the result of formal integration based on the Gauss theorem.

\section{Extension to the \texorpdfstring{$L^p$}{}-scale}\label{extension}

We note that in the context of stochastic processes, the Hilbert space setting is not natural. Indeed, the approriate norm for distributions of random variables is the $L^1$-norm, as the integral over the density yields the total mass. This suggests that we should study
the semigroups $S_\kappa$ not on the space $L^2(\Omega_0)$, but on the space $L^1(\Omega_0)$. Likewise, the proper space on which to consider the semigroups $T_\kappa$ is $L^\infty(\Omega_0)$.

Moreover, one would expect these semigroups to have additional properties such as positivity and contractivity. In this section, we first establish these additional properties. This will allow us to extrapolate our semigroups to the whole $L^p$-scale. Also our convergence results extend to these spaces.

\subsection{Generation results} 

As in Section \ref{sect.gen} we once again write $\fa$ instead of $\fa_1$ and $\fq$ instead of $\fq_1$ to simplify notation. We begin by establishing additional properties of the semigroups $T_2$ and $S_2$ whose existence follows from Corollary \ref{c.l2gen}. 

\begin{prop}\label{p.properties}
The semigroup $T_2$ has the following properties.
\begin{enumerate}
[(a)] 
\item $T_2$  is real, i.e.,\ if $u \in L^2(\Omega_0)$ is real-valued then so is $T_2(t)u$ for all $t\geq 0$;
\item $T_2$ is positive, i.e.,\ if $u \geq 0$ almost everywhere then $T_2(t)u \geq 0$ almost everywhere for all $t\geq 0$;
\item $T_2$ is $L^\infty$-contractive, i.e.,\ if $u\in L^\infty(\Omega_0)$ then for every
$t\geq 0$ we have $T_2(t)u \in L^\infty(\Omega_0)$ and $\|T_2u\|_\infty \leq \|u\|_\infty$.
\end{enumerate}
The semigroup $S_2$ is also real, positive and $L^\infty$-contractive.
\end{prop}

\begin{proof}
The proof of all three parts is based on the Ouhabaz' criterion \cite[Theorem 2.2]{ouh05} and its corollaries.
 We do the necessary calculations for the form 
$\fa$. Similar calculations {for $\fa^*$} yield the corresponding properties for the semigroup $S_2$.\medskip

(a) Obviously, if $u\in D(\fa) = \ha $ 
then also $\Re u, \Im u \in D(\fa)$. 
As all the coefficients in $A$, $c$, $\tau_0, \ldots, \tau_n$ {and $b_{k,\ell}$ ($1\leq k,\ell\leq N$)} are real-valued, 
$\fa[\Re u, \Im u] \in \CR$ for all $u \in D(\fa)$. Now part (a) follows from \cite[Proposition 2.5]{ouh05}. \medskip 

(b) Since $\ha $  
is a lattice with respect to the usual ordering, it follows that if $u\in D(\fa)$ is real-valued then also
$u^+ \in D(\fa)$ {and $u^+u^-=0$ almost everywhere}. Moreover, {by Stampaccia's Lemma \cite[Lemma 7.6]{gt}},
 we have $\partial_j u^+ = (\partial_j u)\one_{\{u>0\}}$ so that, in particular, we have
$(A\nabla u^+)\cdot {\overline{\nabla u^-}} = 0$ almost everywhere. 
Note that we have $(u_{|k})^+= (u^+)_{|k}$ for any $k=0, \ldots, N$, i.e.,\ the trace of the positive part of $u$ is the 
positive part of the trace of $u$. From this it follows that $u_{|k}^+ u_{|k}^-=0$ almost everywhere for $k=0, \ldots, N$. 
Using this we see that for a real-valued $u\in D(\fa)$ we have
\[
\fa [u^+, u^-] = -\sum_{k,\ell \in \mc N} \int_{\Gamma_{k,\ell}} \tau_k {b_{k,\ell}}u_{|\ell}^+u_{|k}^-\, \ud \sigma \leq 0
\]
{as $\tau_k {b_{k,\ell}} \geq 0$ almost surely.}
It now follows from \cite[Theorem 2.6]{ouh05} that $T_2$ is positive.\medskip

(c) In view of \cite[Theorem 2.13]{ouh05} for this part we have to show that if $u\in D(\fa)$ then we also have that
$(1\wedge |u|)\sgn u \in D(\fa)$  and that 
\begin{equation}\label{eq.ouha}
\Re \fa [ (1\wedge |u|)\sgn u, (|u|-1)^+\sgn u] \geq 0.
\end{equation}
Here we have $\sgn z \coloneqq \frac{z}{|z|}$ for a complex $z\neq 0$ and $\sgn 0 \coloneqq 0$.
The condition that $(1\wedge |u|)\sgn u \in D(\fa)$ follows from standard properties of $H^1$-functions, see e.g.\ 
the proof of \cite[Theorem 4.6]{ouh05}. 
As in the proof of that result, we see that
\[
\Re {\fq} [(1\wedge |u|)\sgn u, (|u|-1)^+\sgn u] \geq 0.
\] 
Let us now take care of $\fb$. Writing
\[
I_{k,\ell} \coloneqq \int_{\Gamma_{k,\ell}} \tau_k \big( (1\wedge |u_{|k}|)\sgn u_{|k} - b_{k,\ell}
(1\wedge |u_{|\ell}|)\sgn u_{|\ell}\big) (|u_{|k}|-1)^+\overline{\sgn u_{|k}}\, \ud \sigma
,\]
we have
\begin{align*} 
 \Re \fb  [(1\wedge |u|)\sgn u, (|u|-1)^+\sgn u] = \Re \sum_{\substack{k \in \mc N\\\ell\in \mc N_0}} I_{k,\ell}. \end{align*}
Since the integrand in $I_{k,\ell}$ vanishes on the set $\{|u_{|k}| < 1\}$ we see that $I_{k,\ell}$
equals
\begin{align*}
& \int_{\Gamma_{k,\ell}\cap\{|u_{|k}|\geq 1\}} \tau_k \left [ (1\wedge |u_{|k}|)\sgn u_{{|k}} - b_{k,\ell}(1\wedge |u_{|\ell}|)\sgn u_{|\ell}\right ] (|u_{|k}|-1)^+\overline{\sgn u_{|k}}\, \ud \sigma\\
 &\quad = \int_{\Gamma_{k,\ell}\cap\{|u_{|k}|\geq 1\}} \tau_k \big [ 1- {b_{k,\ell}}(1\wedge |u_{|\ell}|)\, \sgn (u_{|\ell} \,  \overline{u_{|k}})\big ] (|u_{|k}|-1)^+\, \ud \sigma.
\end{align*}
Since $\Re\sgn (u_{|\ell}\overline{u_{|k}}) \in [-1,1]$ and $0\leq b_{k,\ell}\leq 1$, it follows that $\Re I_{k,\ell} \geq 0$ so that, alltogether, we have proved \eqref{eq.ouha}. This finishes the proof.
\end{proof}

We can now extend the semigroups $T_2$ and $S_2$ to the scale of $L^p$-spaces ($1\leq p \leq \infty$). {A family of semigroups $T_p$ $(1\leq p \leq \infty)$, where $T_p$ acts on the space $L^p$, is called \emph{consistent}, if for each choice of
$1\leq p, q \leq \infty$ we have $T_p(t)f= T_q(t)f$ for all $f \in L^p\cap L^q$ and $t\geq 0$.}

\begin{cor}\label{c.lp}
There are consistent families $T_p$ and $S_p$ of contraction semigroups on $L^p(\Omega_0)$ for $1\leq p \leq \infty$.
For $1\leq p <\infty$ these semigroups are strongly continuous whereas $T_\infty$ and $S_\infty$ are merely weak$^*$-continuous. 
Moreover, we have $T_p^* = S_q$ where $\frac{1}{p}+\frac{1}{q}=1$ with the convention that $\frac{1}{\infty}= 0$.
\end{cor}

\begin{proof}
As we have seen, $T_2$ is a contraction semigroup and it follows from Proposition \ref{p.properties} (c) that it restricts to a contraction semigroup $T_\infty$ on $L^\infty(\Omega_0)$. As a consequence of the Riesz--Thorin interpolation theorem (see e.g. \cite[Theorem 6.27]{folland}), $T_2$ restricts to a contraction semigroup $T_p$ on every $L^p(\Omega_0)$ for $2<p<\infty$. 

Let us prove that $T_\infty$ is weak$^*$-continuous. To that end, let $u\in L^\infty(\Omega_0)$ and $t_n \to 0+$. Since $T_2$ is strongly continuous we have $T_{2}(t_n)u\to u$ in $L^2(\Omega_0)$. Passing to a subsequence, we may (and shall) assume that $T_{2}(t_n)u \to u$ pointwise almost everywhere. Since the sequence $\left (T_{2}(t_n)u\right )_{n \ge 1}$ is bounded in the $\|\cdot\|_\infty$-norm it follows from the dominated convergence theorem that
\[
\int_{\Omega_0} \left ( T_2(t_n)u \right ) \bar v\dlam \to \int_{\Omega_0} u \bar v\dlam
\]
for every $v\in L^1(\Omega_0)$. This proves that $T_\infty(t_n)u=T_2(t_n)u$ converges in the weak$^*$-topology to $u$, so that 
$T_\infty$ is weak$^*$-continuous. {We note that $T_\infty$ is never strongly continuous. Indeed, a general result
due to Lotz \cite{lotz} shows that every strongly continuous semigroup on $L^\infty(\Omega_0)$ is automatically uniformly continuous 
and thus has a bounded generator. In our situation this would yield $L^\infty(\Omega_0) \subset \ha$ which is absurd.}

We now turn to continuity of the semigroups $T_p$ for $2<p<\infty$. Let $q$ be the conjugate of $p$.   Since $L^q (\Omega_0) \subset L^1 (\Omega_0)$, and $T_pu=T_\infty u$ for $u\in L^p(\Omega_0)\cap L^\infty(\Omega_0)$
it follows from the above that $T_p(t)u \to u$, as $t\to 0+$, weakly {in $L^p(\Omega_0)$} 
for all $u \in L^p(\Omega_0)\cap L^\infty(\Omega_0)$. As is well known
a weakly continuous semigroup is strongly continuous, see \cite[Theorem I.5.8]{en}. Actually, to use that theorem, we would need to prove that  $t\mapsto T_p(t)u$ is weakly continuous for every $u \in L^p(\Omega)$. However, inspection of the proof shows that for a bounded semigroup it actually suffices to prove weak continuity of the orbits for $u$ in a dense subset. This shows that $T_p$ is strongly continuous.

The same argument yields consistent semigroups $S_p$ for $2\leq p \leq \infty$ where $S_p$ is strongly continuous
and $S_\infty$ is weak$^*$-continuous.\smallskip

We next prove that $T_\infty$ is an adjoint semigroup. To that end, let $u\in L^2(\Omega_0)$ and $v\in L^\infty(\Omega_0)$. Since
$S_2^* = T_2$ we find
\[
\Big|\int_{\Omega_0} (S_2(t)u) \bar v \dlam\Big| = \Big|\int_{\Omega_0} u\overline{T_2(t)v}\dlam\Big|
\leq \|u\|_1 \|T_2(t) v\|_\infty \leq \|u\|_1\|v\|_\infty.
\]
Taking the supremum over $v \in L^\infty(\Omega_0)$ with $\|v\|_\infty\leq 1$, we see that
$S_2(t)u \in L^1(\Omega_0)$ and $\|S_2(t)u\|_1 \leq 1$. Thus, $S_2(t)$ can be extended to a contraction
$S_1(t)$ on $L^1(\Omega_0)$. Clearly, $[S_1(t)]{^*} = T_\infty(t)$ which proves that $T_\infty$ consists of adjoint operators. It follows
from the weak$^*$-continuity of $T_\infty$ that the orbits of $S_1$ are weakly continuous hence, by \cite[Theorem I.5.8]{en},
$S_1$ is a strongly continuous semigroup. Similarly, $T_2$ extends to a strongly continuous contraction semigroup $T_1$ on $L^1(\Omega_0)$
with $T_1^* = S_\infty$.

Finally, in an analogous way we get $S_p = T_q^*$ and $T_p = S_q^*$ for $1<p<2$ where $q\in (2,\infty)$ is such that $\frac{1}{p}+\frac{1}{q} =1$. It follows that all semigroups $T_p $ (and $S_p$), $ p \in (1,\infty)$ are weakly continuous, and hence also strongly continuous.
\end{proof}

\subsection{Convergence results}

Applying Corollary \ref{c.lp} for every $\kappa$, we obtain for every $p \in [1,\infty]$ families $T_{p,\kappa}$ and $S_{p,\kappa}$
of semigroups. For $p=2$, convergence of these semigroups was established in Section \ref{sect.conv}. Let us note that 
the space $H_{0}$ which appears in the limit semigroup is contained in $L^p(\Omega_0)$ for every $1\leq p \leq \infty$. Moreover, 
the right hand side of \eqref{eq.projection} is well defined also for $u$ in $L^p(\Omega_0)$ and defines a projection on
 $L^p(\Omega_0)$ with range $H_{0}$. By slight abuse of notation, we denote that projection still by $P_{H_0}$. Thus, we may view $\e^{t(Q-C)}P_{H_0}$ and $\e^{t(Q^*-C)}P_{H_0}$ as degenerate semigroups on $L^p(\Omega)$. The main result of this section (which, along with Theorem \ref{t.l2conv}, is the main result of the paper as well) extends Theorem \ref{t.l2conv} to the setting of $L^p$ spaces.  

\begin{thm}\label{t.lpconv}
For $1\leq p < \infty$, $t>0$ and $u \in L^p(\Omega_0)$ we have 
\[ \grak T_{p,\kappa}(t)u = \e^{t(Q-C)}P_{H_0} u \quad \mbox{and} \grak S_{p,\kappa}(t) u = \e^{t(Q^*-C)}P_{H_0} u\]  in the
$L^p(\Omega)$ norm. Moreover,
for $u\in  L^\infty(\Omega_0)$  we have 
\[\grak T_{\infty,\kappa}(t)u = \e^{t(Q-C)}P_{H_0} u  \quad\mbox{and}\quad \grak S_{\infty, \kappa}(t)u = \e^{t(Q^*-C)}P_{H_0} u\] 
in the weak$^*$ topology of $L^\infty(\Omega_0)$.
\end{thm}

\begin{proof}
Let $u\in L^\infty(\Omega_0)$. By the previous theorem, $T_{2,\kappa}(t) u$ converges to $\e^{t(Q-C)}P_{H_0} u$ in the norm of $L^2(\Omega_0)$. Passing to a subsequence, we may and shall assume that we have almost sure convergence.
Since the sequence $T_{2,\kappa}(t)u$ is uniformly bounded (by $\|u\|_\infty$) it follows from the dominated convergence theorem
that $T_{2,\kappa}(t)u$ converges to $\e^{t(Q-C)}P_{H_0}u$ weak$^*$ in $L^\infty(\Omega_0)$. Another consequence of the dominated convergence theorem is that $T_{2,\kappa}(t)u = T_{p,\kappa}(t)u \to \e^{t(Q-C)}P_{H_0} u$ in $L^p(\Omega_0)$ for every $1\leq p <\infty$. Since $L^\infty(\Omega_0)$ is 
dense in $L^p(\Omega_0)$ for $1\leq p <\infty$ and since the operators $T_{p,\kappa}(t), \kappa >0 $ are uniformly bounded, a
$3\eps$ argument yields that $T_{p,\kappa}(t)u  \to \e^{t(Q-C)}P_{H_0}f$ in $L^p(\Omega_0)$ for every $u\in L^p(\Omega_0)$. The corresponding
statements for $S_{p,\kappa}$ are obtained similarly.
\end{proof}

\begin{rem}
We note that with Theorem \ref{t.lpconv} at hand, we can now also generalize Corollary \ref{cor.main} to the $L^p$-setting for
$p \in [1,\infty)$ with basically the same proof.
A related convergence result also holds for $p=\infty$. However, the situation is slightly more complicated as we are dealing with a semigroup which is not strongly continous. However, if we interpret the integral in \eqref{eq.voc} as a weak$^*$-integral, then the integral is well-defined whenever $f$ is weak$^*$ measurable and $\|f\|$ is integrable. If we accept \eqref{eq.voc} as definition of a mild solution in the case of $p=\infty$, then, in the situation of Corollary \ref{cor.main} we easily obtain pointwise weak$^*$-convergence of mild solutions.
\end{rem}


\section{Discussion}

In modeling biological processes one often needs to take into account different time-scales of the processes involved \cite{banalachokniga,knigazcup}. This is in particular the case when one of the components of the model is diffusion which in certain circumstances may transpire to be much faster than other processes. For example, 
in the Alt and Lauffenberger's \cite{alt} model of leucocytes reacting to a bacterial invasion by moving up a gradient of some chemical  attractant produced by the bacteria (see Section 13.4.2 in \cite{keenersneyd}) a system of three PDEs is reduced to one equation provided bacterial diffusion is much smaller than the diffusion of leukocytes or of chemoattractants (which is typically the case). Similarly, in the early carcinogenesis model of Marcinak--Czochra and Kimmel \cite{osobliwiec,nowaania,markim1,markim2},  a system of two ODEs coupled with a single diffusion equation (involving Neumann boundary conditions) is replaced by a so-called shadow system of integro-differential equations with ordinary differentiation, provided diffusion may be assumed fast. 

In this context it is worth recalling that one of the fundamental properties of diffusion in a bounded domain is that it `averages' solutions (of the heat equation with Neumann boundary condition) over the domain. As it transpires, it is this homogenization effect of diffusion, when coupled with other physical or biological forces that leads to intriguing singular perturbations; this is exemplified by the analysis of the models in Section \ref{sec:6} (see also \cite{osobliwiec}).

In this paper we describe the situation in which fast diffusion in several bounded domains separated by semi-permeable membranes is accompanied by low permeability of the membranes. Assuming that the flux through the membranes is of moderate value, we show that such models are well-approximated by those based on Markov chains. More specifically, because of the homogenization effect, all points in each domain of diffusion are lumped into a single state, and the non-negligible  flux forces so-formed new states to communicate as the states of a Markov chain (eq. \eqref{qkl} provides the entries in its intensity matrix). 

Certainly, applicability of the theorem depends in a crucial way on whether and to what extend diffusion involved in the model is faster than other processes. Nevertheless, the literature of the subject provides numerous examples of such situations. Two of them: the model of intracellular dynamics and that of neurotransmitters are discussed in detail in Section \ref{sec:6}. Our third example is of slightly different type: its main purpose is to  show that diffusion of kinases in a cell cannot be too fast for signaling pathways to work  properly.

From the mathematical viewpoint, the established principle  is a close relative of the famous Freidlin-Wentzell averaging principle 
(\cite{fw,fwbook}, see also \cite{freidlin}), but it differs from its more noble cousin 
in the crucial role played by transmission conditions, which are of marginal or no importance in the latter. These conditions, sometimes referred to as radiation boundary conditions, describe in probabilistic and analytic terms the way particles permeate through  the membranes, and thus, indirectly, the flux, which influences the model in a critical way (see eq. \eqref{qkl} again).  

\appendix
\section{Proof of Lemma \ref{l.intersect}}

In this appendix, we prove Lemma \ref{l.intersect}, which states that the set of points in the boundary which are adjacent to three or more of the subdomains $\Omega_j$, $j=0, \ldots N$ has Hausdorff measure zero.

\begin{proof}[Proof of Lemma \ref{l.intersect}]
Since any set of non-zero Hausdorff measure contains a non-empty open set, it suffices to show 
that if there is a non-empty open set $V \subset \Gamma_{k,\ell}$, then there is a non-empty  subset $V_0$ of  $V$ which is open {(in the relative topology)} in 
$\Gamma_{k,\ell}$ and an open subset $U$ of $\CR^{d}$ such that 
\begin{equation} \label{warunki} 
V_0 \subset U, \qquad U \setminus V_0 \subset \Omega_k \cup \Omega_\ell. 
\end{equation}

To prove this claim, pick a point in $V$. Choosing a suitable coordinate system, we may assume without loss of generality that there exists an open neighborhood $V_0 \subset V$ of this point such that the following conditions hold (see Figure \ref{empty}):
\begin{figure}[h]
\begin{tikzpicture}
\draw[dashed,color=blue] (-1,0)--(1,0)--(1,5.2)--(-1,5.2)--(-1,0); 
\draw[dashed,color=orange] (-2+0.5,0)--(2+0.5,4)--(1+0.5,5)--(-3+0.5,1)--(-2+0.5,0);
\draw (-1,2.5)--(-0.5,2.8)--(-0.25,1.95)--(0.25,1.9)--(1,2.5); 
\draw[dashed,color=green] (-1,2.5)--(-1,0)--(1,0)--(1,2.5);
\draw[dashed,color=magenta] (-1,2.5)--(1.5,5)--(2.5,4)--(1,2.5);

\node[above] at (0,0.5) {$C_k^-$};
\node[above] at (0,4) {$C_k^+$};
\node[above] at (1.8,3.75) {$C_\ell^-$};
\node[above] at (-1.5,0.75) {$C_\ell^+$};
\end{tikzpicture}
\caption{The set $C_k^+ \cap C_\ell^+$ is empty: $V_0$ is the graph drawn with solid line.}\label{empty}
\end{figure}
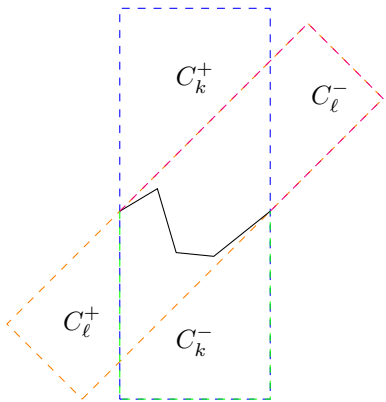 

\begin{itemize}
\item[1. ] There are two cylinders $C_i=B_i \times (a_i,b_i), i=k,{\ell}$, where $B_i$'s are open balls in $\CR^{d-1} $ and $(a_i,b_i)${'}s are open subintervals of $\CR$, and

\item[2. ] there is an isomorphism $\mcj $ of $\CR^d$ and Lipschitz continuous functions $g_i: B_i \to \CR $ such that defining $\phi_i (w,t) = t- g_i (w)$ for $(w,t) \in C_i$, we have 
\begin{itemize}
\item [a. ] $\Omega_k \cap C_k= \{\phi_k < 0\}, C_k \setminus \overline{\Omega_k} = \{\phi_k > 0\},$  and $V_0 = \{\phi_k=0\},$
\item [b. ] $\Omega_\ell \cap C_\ell = \mcj\{\phi_\ell < 0\}, C_\ell \setminus \overline{\Omega_\ell} = \mcj \{\phi_\ell > 0\},$  and $V_0 = \mcj \{\phi_\ell=0\}.$
\end{itemize}
\end{itemize}
We note that continuity of $g_i$s implies continuity of $\phi_i$s as functions of two variables.

We claim that the neighborhood we look for is $U = C_{{k}}\cap C_{{\ell}}$. Since {$V_0\subset C_k$ and $V_0\subset C_\ell$, we clearly have $V_0 \subset U$. Let us show that $U\setminus V_0 \subset \Omega_k\cup\Omega_\ell$}.
To this end, we first simplify our notations by putting
\begin{align*} C_k^+ &= \{\phi_k > 0\}, \quad \,\, \, \, C_k^- = \{\phi_k < 0\}, \\
 C_\ell^+ &= \mcj \{\phi_\ell > 0\}, \quad C_\ell^- = \mcj\{\phi_\ell < 0\},\end{align*} 
and then write $U\setminus V_0 = (C_k^- \cup C_k^+)\cap (C_\ell^- \cup C_\ell^+)$ as the union of 
\[ C_k^- \cap (C_\ell^- \cup C_\ell^+) \subset C_k^- \subset \Omega_k \]
and 
\[ C_k^+ \cap (C_\ell^- \cup C_\ell^+) = (C_k^+ \cap C_\ell^-) \cup (C_k^+ \cap C_\ell^+).\]
Since $C_k^+ \cap C_\ell^- \subset C_\ell^- \subset \Omega_\ell$, it suffices to show that 
$C_k^+ \cap C_\ell^+$ is empty (see again Figure \ref{empty}). 

Suppose that, contrary to our claim, there is $(w_0,t_0) \in C_k$ that belongs to $C_k^+ \cap C_\ell^+$. Then, there is $\eps >0$ such that $(w,t_0) \in C_k^+ \cap C_\ell^+$ for all $w {\in B(w_0,\eps)}$, {where $B(w_0,\eps)$ denotes the ball in $\CR^{d-1}$ of radius $\eps$ centered at $w_0$}. 
Fix such a $w$, and let $I \subset C_k$ be the closed {line segment} 
with ends $z_1 = (w,g_k(w))$ and $z_2 = (w,t_0).$ Then $\tilde I := 
\mcj^{-1} I$ is {also a line segment} (since $\mcj$ is an isometry) {and it is} contained in $C_\ell $ (since $C_\ell $ is a convex set containing ${\mcj}^{-1}(z_1)$ and ${\mcj}^{-1}(z_2)$). {Moreover,} we have $\phi_\ell \circ \mcj^{-1} (z_1) =0$ and $\phi_\ell \circ \mcj^{-1} (z_2) >0$.

We note that, {$\phi_\ell$ is positive on $\tilde I\setminus\mcj\{z_1\}$. Indeed, otherwise we would have $\phi_\ell (z) =0$ for some $z$ in the interior of $\tilde I$. But this implies that $\mcj (z) \in I$ and, since $\mcj V_0 = V_0$, also
$\mcj (z) \in V_0$. This is a contradiction to the } fact that, by the definition of $C_k$, on $I$ there is precisely one point, {namely} $z_1$, of $V_0$. Therefore, $I$ does not contain points of $\Omega_\ell$. On the other hand, by the definition of $C_k$, the open {line segment} joining $z_0 = (w,a_k)$ and $z_1$ is contained in $\Omega_k$ and thus cannot contain points of $\Omega_\ell$, either.

{Altogether we have} showed that the cylinder 
\[ C= {B}(w_0, \eps) \times (a_k, t_0) \]
has empty intersection with $\Omega_\ell$. {But $C$} contains $V_0$. This clearly contradicts the fact that $V_0$ is a part of boundary of $\Gamma_\ell$, and {hence} our assumption that $C_k^+ \cap C_\ell^+$  is non-empty was false. 
\end{proof}

We note that the argument presented above does not require the boundary to be Lipschitz. It suffices to assume that the boundary is continuous.

\end{document}